\documentclass[numbers,webpdf]{ima-authoring-template}%

\graphicspath{{Fig/}}

\theoremstyle{thmstyletwo}%
\newtheorem{theorem}{Theorem}%
\newtheorem{lemma}[theorem]{Lemma}%

\newtheorem{remark}{Remark}%

\numberwithin{equation}{section}

\usepackage{mathtools}

\makeatletter
\DeclareRobustCommand*\cal{\@fontswitch\relax\mathcal}
\makeatother

\def\e{\varepsilon}

\def\cF{{\cal F}}

\def\cD{{\cal D}}
\def\cP{{\cal P}}
\def\cJ{{\cal J}}

\def\R{\mathbb{R}}

\def\T{{\relax\ifmmode I\!\!\hspace{-1pt}T\else$I\!\!\hspace{-1pt}T$\fi}}

\def\lsim{\raisebox{-1ex}{$~\stackrel{\textstyle<}{\sim}~$}}

\def\cS{{\cal F}}

\def\cD{{\cal D}}

\def\cJ{{\cal J}}

\def\cN{{\cal N}}

\def\cX{{\cal X}}
\def\cS{{\cal S}}
\def\cP{{\cal P}}

\newcommand{\be}{\begin{equation}}
\newcommand{\ee}{\end{equation}}
\newcommand{\bea}{$$ \begin{array}{lll}}
\newcommand{\eea}{\end{array} $$}

\def \exp{\mathop{\rm    exp}}

\newcommand\dist{\mathop{\rm dist}}

\newcommand\eref[1]{{\rm (\ref{#1})}}

\newcommand\norm[1]{\left\|#1\right\|}

\begin{document}

\DOI{DOI HERE}
\copyrightyear{2025}
\vol{00}
\pubyear{2025}
\access{Advance Access Publication Date: Day Month Year}
\appnotes{Paper}
\copyrightstatement{}
\firstpage{1}

\title[Optimal Recovery Meets Minimax Estimation]{Optimal Recovery Meets Minimax Estimation}

\author{Ronald DeVore\ORCID{0009-0005-1105-5925}
\address{\orgdiv{Department of Mathematics}, \orgname{Texas A\&M University}}}
\author{Robert D. Nowak\ORCID{0000-0001-8743-0887}
\address{\orgdiv{Department of Electrical and Computer Engineering}, \orgname{University of Wisconsin--Madison}}}
\author{Rahul Parhi*\ORCID{0000-0002-1971-7699}
\address{\orgdiv{Department of Electrical and Computer Engineering}, \orgname{University of California, San Diego}}}
\author{Guergana Petrova\ORCID{0000-0002-8873-4150}
\address{\orgdiv{Department of Mathematics}, \orgname{Texas A\&M University}}}
\author{Jonathan W. Siegel\ORCID{0000-0002-1493-4889}
\address{\orgdiv{Department of Mathematics}, \orgname{Texas A\&M University}}}

\authormark{DeVore, Nowak, Parhi, Petrova, and Siegel}

\corresp[*]{Corresponding author: \href{mailto:rahul@ucsd.edu}{rahul@ucsd.edu}}

\received{Date}{0}{Year}
\revised{Date}{0}{Year}
\accepted{Date}{0}{Year}

\abstract{A fundamental problem in statistics and machine learning is to estimate a function $f$ from possibly noisy observations of its point samples. The goal is to design a numerical algorithm to construct an approximation $\hat f$ to $f$ in a prescribed norm that asymptotically achieves the best possible error (as a function of the number $m$ of observations and the variance $\sigma^2$ of the noise). This problem has received considerable attention in both nonparametric statistics (noisy observations) and  optimal recovery (noiseless observations). Quantitative bounds require assumptions on $f$, known as \emph{model class} assumptions. Classical results assume that $f$ is in the unit  ball of a Besov space. In nonparametric statistics, the best possible performance of an algorithm for finding $\hat f$ is known as the minimax rate and has been studied in this setting under the assumption that the noise is Gaussian. In optimal recovery, the best possible performance of an algorithm is known as the optimal recovery rate and has also been determined in this setting.  While one would expect that the minimax rate recovers the optimal recovery rate when the noise level $\sigma$ tends to zero, it turns out that the current results on  minimax rates do not carefully determine the dependence on $\sigma$ and  the limit cannot be taken. This paper handles this issue and determines the \emph{noise-level-aware} (NLA) minimax rates for Besov classes when error is measured in an $L_q$-norm with matching upper and lower bounds. The end result is a reconciliation between minimax rates and optimal recovery rates. The NLA minimax rate continuously depends on the noise level and recovers the optimal recovery rate when $\sigma$  tends to zero.
}

\maketitle

\section{Introduction}
\label{intro}
Let $\Omega \subset \R^d$ be a bounded domain (open, bounded, simply connected set) with a Lipschitz boundary and  $\overline{\Omega}:= \Omega \cup \partial\Omega$ be the closure of $\Omega$. 
A fundamental problem in statistics and machine learning is to construct an   approximation $\hat{f}$ to a continuous function $f: \overline{\Omega}\to \R$ when given the $m$  noisy observations $y_i$ at the data sites $x_i\in\overline \Omega$, 
\begin{equation} \label{eq:observations}
    y_i = f(x_i) + \eta_i, \quad i = 1, \ldots, m,
\end{equation}
where  $\eta_i$, $i = 1, \ldots, m$, are i.i.d.\ Gaussian random variables with mean $0$ and variance $\sigma^2$.
We will evaluate the quality of the approximant  $\hat f$ by  measuring  the 
error $\|f-\hat f\|_{L_q(\Omega)}$  in 
an $L_q(\Omega)$-norm.  Notice that both $\hat f$ and the error  are random variables depending on the noise.
 
In order to obtain bounds on $\|f-\hat f\|_{L_q(\Omega)}$,  one  needs additional information on the data generating function $f$.  This additional information is the fact  that $f$ lies in a \emph{model class} $K$ which is a compact subset of $C(\overline{\Omega})$. In this paper, we consider $K$ to be the unit ball of a Besov space $B^s_\tau(L_p({\Omega}))$.  For those unfamiliar with Besov spaces, this space   can be intuitively viewed as the space of functions with $s > 0$ derivatives in $L_p({\Omega})$, where $0 < p \leq \infty$,  and $\tau$,   $0 < \tau \leq \infty$, allows for additional finer control of the regularity of the underlying functions (see \S \ref{S:Besov} for the definition and properties of these spaces). The unit ball of $B^s_\tau(L_p({\Omega}))$ compactly embeds into $C(\overline{\Omega})$ if and only if 
\begin{equation}
\label{compactembedding}
s>d/p.
\end{equation}
Therefore, we assume \eqref{compactembedding} in the remainder of this paper. To simplify the presentation, throughout this paper, we will only consider the case when $\Omega = (0,1)^d$ is the unit cube in $\mathbb{R}^d$.  We leave the  extension of our results to more general domains for future work.

Given a set $\cX$ of $m$  data sites 
\begin{equation}
\cX:=\{x_1, x_2, \ldots, x_m\}\subset \overline \Omega,
\end{equation}
we consider   an algorithm $A$ (for numerically creating an approximation  $\hat f$) to be a mapping 
 \begin{equation}
    A: \R^m \to L_q(\Omega),
\end{equation}
that sends  the observations \eqref{eq:observations} to an approximation $\hat{f} = A(y)$ of $f$, where $y=(y_1, \ldots, y_m)$.    Notice that $\hat f$ is an $L_q(\Omega)$-valued random variable.  The performance of the algorithm is usually measured via the worst case $L_q(\Omega)$-risk 
\be
    E_A(K;\sigma,\cX)_q := \sup_{f \in K} \: \mathbb E\norm{f - A(y)}_{L_q(\Omega)}.
\ee
Here the expectation $\mathbb E$ is over the random noise $\eta_i$ in our measurements   $y_i$, $i=1,\dots,m$. Further, we define the
\emph{minimax risk} to be the optimal worst-case performance over all possible algorithms and choices of $\cX$ with $|\cX|=m$, i.e.,
\be
\label{MX}
R_m(K;\sigma)_q:={\inf_{A,\cX} E_A(K;\sigma,\cX)_q.
}
\ee
The special case when $\sigma=0$, i.e., the case of no noise, is the problem of \emph{optimal recovery} (OR) and has been well studied in the mathematics community.

The most celebrated result for this minimax problem is the one  of Donoho and Johnstone~\cite{DJ} who study the case   $\Omega=(0,1)$, $d=1$, $q=2$, and $m$ equally spaced  data sites $\{x_i\}_{i=1}^m$ from $\Omega$.   They prove that a certain algorithm, known as wavelet shrinkage, is asymptotically minimax optimal for the estimation of functions that lie in the unit ball $K$ of the Besov spaces $B_\tau^s(L_p(0,1))$,   $1 \leq \tau,p \leq \infty$, and $s > 1/p$. Moreover, they prove   that for fixed $\sigma>0$, the asymptotic decay of the minimax error in $L_2$ is
\be \label{DJminimax}
R_m(K;\sigma)_2\asymp m^{-\frac{s}{2s+1}}, \quad m\ge 1.
\ee
On the other hand, the OR rate in this case is well known to  behave asymptotically as (see \cite{BDPS,NT,KNS})
\be \label{ORrate}
R_m(K;0)_2\asymp m^{-s+(1/p-1/2)_+}, \quad m\ge 1.
\ee

The goal of this paper is to understand the discrepancy between the minimax and OR rates.  In particular, we note that the rate decay \eref{DJminimax} needs to be interpreted as follows. For each fixed $\sigma>0$, the constants of equivalency in \eref{DJminimax} depend on
$\sigma$. This precludes letting $\sigma$ tend to zero to obtain the   recovery rate.  On the other hand, a special case of the results in this paper gives the rate
\be 
\label{newminimax}
R_m(K;\sigma)_2\asymp m^{-s+(1/p-1/2)_+} + \left[\frac{\sigma^2}{m}\right]^{\frac{s}{2s+1}}, \quad m\ge 1,
\ee
where now the constants of equivalency depend on $s,p,\tau$, but do not depend on $\sigma$ or $m$.  This new asymptotic noise-level-aware result allows one to recover the OR rate by letting $\sigma\to 0$.   

There are two components to the minimax (or OR) theory. The first one is to introduce  an algorithm and provide an upper bound for the performance of this algorithm.
In the above setting, the algorithm typically used to prove upper bounds is wavelet shrinkage.  That is, one uses the fact that each $f\in K$ has a wavelet decomposition and the membership in the Besov spaces can be exactly described by a suitable weighted $\ell_p$-norm on these coefficients.  One then uses the data observations to compute
empirical noisy wavelet coefficients and thresholds these coefficients to obtain $\hat f$. An issue that needs to be addressed in this approach is that when the smoothness $s$ is greater than one, it is necessary to use smooth
wavelets.  The wavelet decomposition in this case needs to be altered  near the boundary of $\Omega$.  This part of the analysis is cumbersome to carry out in detail.

The second part of the theory is to provide for any algorithm  and any $m$ data sites lower bounds that match the upper bounds obtained by wavelet thresholding. In the statistics literature, lower bounds are usually proved using information-theoretic methods such as the Le Cam and Fano methods which are based on the Kullback--Leibler (KL) divergence (see, e.g.,  \cite{barron1999risk,yang1999information,AN,DJ,DeJu}). 
Note that  the existing lower bounds do not apply simultaneously for all $\sigma$ and $m$ and  do not consider general data sites. In \S\ref{S:lowerbounds}, we prove our lower bounds (see Theorem~\ref{T:mainexplb}) in a self-contained manner that does not require information theory.

The Donoho--Johnstone result set the stage for a multitude of follow-up works (see, e.g., the overview article~\cite{maxiset} or the books~\cite{hardle2012wavelets,gine2016mathematical} and the references therein). 
 Several of these works (see, e.g.,~\cite{donoho1994ideal,DJ,DJKP,DeJu,maxiset,chambolle1998nonlinear,johnstone1999wavelets,MinimaxMaxiset,donoho1996density,hardle2012wavelets}),  concentrated on two main issues.  The first was to extend the theory to functions of several variables defined on domains $\Omega\subset \R^d$. The second issue  was to allow the recovery error to be measured in a general $L_q$-norm, $1\le q\le \infty$. 
 In particular, we highlight the work of~\cite{DeJu}, where  the authors generalized the result of Donoho and Johstone to Besov spaces defined on $\Omega=(0,1)^d$. They provide upper bounds for the performance of    wavelet shrinkage   for the $L_q$-estimation, $1\le q\le \infty$, of functions that lie in unit ball of $B_\tau^s(L_p(\Omega))$ with $0 \leq \tau, p \leq \infty$ and $s > d/p$.   For example, they show that  
 \be \label{DJminimaxupper}
R_m(K,\sigma)_2\le R_m(K,\cX_m,\sigma)_2\le C(\sigma)  m^{-\frac{s}{2s+d}}, \quad m\ge 1,
\ee
 when the $m$  sample sites $\cX_m$ lie on a tensor product grid of $\Omega$.

   The authors of~\cite{DeJu}  show that  the   rate in the upper bound \eref{DJminimax} is not always $m^{-\frac{s}{2s+d}}$ when $q \neq 2$, but  depends on the relationship between $s$, $p$, $q$, and $d$ (see \S \ref{S:concludingremarks}). However, it turns out (cf.,~\cite[Theorem~4]{DeJu}) that when these parameters satisfy the inequality
\begin{equation}
\label{parameters} 
 q<p+2\frac{sp}{d},
\end{equation}
the $L_q(\Omega)$ upper bound rate remains $m^{-\frac{s}{2s+d}}$.  Note that since $s>d/p$,
this is always the case when $q\le 2$.  We call the case
where \eref{parameters} holds the {\it primary case}.  For simplicity of the presentation in this paper, we restrict ourselves to this primary case when going forward.
 Later, in the section on concluding remarks, we explain what happens when the parameters are not in the primary case.

We reiterate again that the existing minimax bounds, for example the ones in  \cite{DJ,DeJu},  must be interpreted for a \emph{fixed noise level}.
Indeed, the current state of the literature determines the minimax rates asymptotically as $m\rightarrow \infty$ with $\sigma$ fixed. If one wants a bound for all $m\geq 1$, then (undetermined) constants depending on $\sigma$ must be introduced. For example, if \eqref{parameters} holds, the existing results imply that the minimax risk satisfies
\begin{equation}
    c(\sigma) m^{-\frac{s}{2s + d}} \leq  R_m(K;\sigma)_q\leq   C(\sigma) m^{-\frac{s}{2s + d}},\quad m\ge 1,
    \label{eq:MM-rate-intro}
\end{equation}
for every fixed $\sigma > 0$. This does not capture the effect of the noise level in the estimation error rate and instead, hides the effect of noise in the constants $c(\sigma), C(\sigma) > 0$. In fact, prior to the present paper, the exact dependence of these constants on $\sigma$ was  not well-understood.

Furthermore, one would expect, intuitively, that it should be easier to estimate the function when the noise level is small. It turns out that this intuition is correct and is revealed by the OR rate. In OR, the problem of interest corresponds to the case when $\sigma = 0$ (i.e., no noise). In that case, it is known that the OR rate satisfies the bounds
\begin{equation}
    c\, m^{-\frac{s}{d} + \left( \frac{1}{p} - \frac{1}{q}\right)_+} \leq  R_m(K;\sigma = 0)_q\leq   C m^{-\frac{s}{d} + \left( \frac{1}{p} - \frac{1}{q}\right)_+},
    \label{eq:OR-rate-intro}
\end{equation}
where $c, C > 0$ depend only on $s$, $p$, $q$, $\tau$, and $d$. For a proof of \eqref{eq:OR-rate-intro} in our specific setting, we refer the reader to~\cite{BDPS}. Other relevant references for the OR rate for Besov spaces include~\cite{NT,V, KNS} who treat various cases of $q$, $\Omega$, and the positioning of the data sites $\cX$. Comparing \eqref{eq:MM-rate-intro} and \eqref{eq:OR-rate-intro}, we see that the current minimax bounds do not recover the OR rates  as $\sigma \to 0$.

\subsection{Main Contributions of This Paper}
The motivation for this paper is to remove the gap between the minimax rate \eqref{eq:MM-rate-intro} and the OR rate \eqref{eq:OR-rate-intro}.   Our contributions are four-fold.
\begin{enumerate}
    \item We propose an algorithm for using the noisy data \eref{eq:observations}  to approximate functions from  Besov spaces based on piecewise polynomials and a certain thresholding procedure. This formulation is attractive in that it avoids the use of wavelets and having to deal with the proper definition and analysis of boundary-adapted wavelets (see, e.g.,~\cite{CDD,CDV}).

    \item We analyze the performance of our proposed algorithm and carefully keep track of the dependence on $\sigma$.    We provide an   $L_q(\Omega)$-performance bound in probability for  $q\geq 1$ when the parameters are in the primary case (see \eqref{parameters}). This bound is novel for three reasons:
    \begin{enumerate}
        \item To the best of our knowledge, this is the first bound to characterize the error of estimating a function from a Besov space \emph{in probability}.
        
        \item The proof of this result relies on some new results regarding the thresholding of vectors polluted with additive Gaussian noise, which may be of independent interest.

        \item This is the first bound to characterize the error of estimating a function from a Besov space that recovers the OR rate when $\sigma \to 0$.
    \end{enumerate}

    \item We readily derive an $L_q(\Omega)$-error bound in expectation from our probability bound that carefully keeps track of the dependence on $\sigma$. We then derive a   matching lower bound for the expected $L_q(\Omega)$-error that shows the precise dependence on $\sigma$. 
    
    \item Recently, the problem of deriving \emph{noise-level-aware} (NLA) minimax rates in various settings has been considered. For example,  minimax estimation of model classes consisting of H\"older functions (that is, smoothness measured in $L_\infty$) has been studied in~\cite{morkisz2020complexity}. In~\cite{10.1214/24-AOS2446}, the authors study model classes in Hilbert spaces defined by elliptical constraints. In~\cite{10374243,guo2024note,ghosh2025signal}, estimation of sparse vectors has been considered. To the best of our knowledge, our results are  the first NLA minimax rates which hold for general Besov classes with error measured in general $L_q$-norms. Our NLA minimax rate provides a rigorous justification to the intuitive fact that the estimation problem is easier in a low-noise regime. 
\end{enumerate}

We next summarize the main results of this paper.  We first state our results for the case $0<p\le q$ and $q\geq 1$, and later mention how one trivially derives the corresponding results for $p>q$ from these.
\begin{theorem} \label{T:mainprob}
    Let $\Omega = (0,1)^d$ and $1 \leq q < \infty$, $0< \tau \leq \infty$, $0 < p \leq q$ and $s > 0$ be parameters satisfying \eqref{compactembedding} and \eqref{parameters}. Assume that we observe noisy function values according to \eqref{eq:observations} as follows: 
    \begin{enumerate}
        \item The observation points lie on the tensor-product grid 
    \be 
    \label{Gn}
    G_n := \{0, 2^{-n}, \ldots, 1-2^{-n}\}^d, 
    \quad n\geq 1,
    \ee 
    with  total number of observations 
     $m = 2^{nd}.$ 
     \item The noise $\eta_i$, $i = 1, \ldots, m$, are i.i.d.\ Gaussian random variables with mean $0$ and variance $\sigma^2$ for some $\sigma \geq 0$.
    \end{enumerate}
    Then, for any 
    \be
    \label{alpha}
    0 < \alpha <  2-\frac{d(q-p)_+}{sp},
    \ee
    there exists an algorithm $A$ (that depends on $s$, $p$, $q$, $d$, $m$, $\alpha$, and $\sigma$) such that for any $f \in U(B^s_\tau(L_p(\Omega)))$ and $t\geq 0$, we have the bound
    \begin{equation}
    \label{T1}
        \mathbb P\left( \|f - A(y) \|_{L_q(\Omega)} \ge C\left[m^{-\frac{s}{d} + \left( \frac{1}{p} - \frac{1}{q}\right)_+} + t \left[\frac{\sigma^2}{m}\right]^{\frac{s}{2s + d}}\right] \right) \leq C e^{-ct^\alpha}.
    \end{equation}
     Here the constants $C$ and $c$ depend only upon $s,p,q,d$ and $\alpha$, but not on $m$ or $\sigma$.
\end{theorem}
\begin{remark}
Note that there always exists $\alpha$ satisfying  condition \eqref{alpha}  since we are in the primary regime. We also remark that the assumption on the noise being i.i.d. Gaussian can be relaxed. Indeed, Theorem \ref{T:mainprob} holds more generally for independent sub-Gaussian noise with variance bounded by $\sigma^2$ for each measurement.   
\end{remark}

The probability bound in Theorem~\ref{T:mainprob} readily implies an upper bound on the expectation of the error (see Theorem \ref{T:mainexp} below). It turns out that the upper bound  is also sharp. Indeed, in 
\S\ref{S:lowerbounds} we prove the following lower bound.

\begin{theorem}  \label{T:mainexplb} 
    Consider the model class $K = U(B^s_\tau(L_p(\Omega)))$, where $\Omega = (0,1)^d$ and $0 < p,\tau \leq \infty$,   $1\leq q\leq \infty$, and $s > 0$, are parameters satisfying \eqref{compactembedding} and \eref{parameters}. 
    Then, we have
    \begin{equation} \label{mmbound1}
R_m(K;\sigma)_q\geq C\left(m^{-\frac{s}{d} + \left(\frac{1}{p} -\frac{1}{q}\right)_+}+ \min\left\{1, \left[\frac{\sigma^2}{m}\right]^{\frac{s}{2s + d}}\right\} \right),
    \end{equation}
    where the  constants $C$ depends only on $s$, $p$, $q$, $\tau$, and $d$. 
\end{theorem}

The transformation of our probabilistic results into expectation results, gives the following theorem.

\begin{theorem}  \label{T:mainexp} 
    Let $\Omega = (0,1)^d$, and  $q\ge 1$, $0< \tau,p \leq \infty$,  and $s>0$, where the parameters $s$, $p$, and $q$ satisfy~\eqref{compactembedding} and \eqref{parameters}.  For the model class $K = U(B^s_\tau(L_p(\Omega)))$,  there exist constants  depending on
    $s$, $p$, $\tau$, and $q$, but not on $\sigma$ or $m$, such that
    \begin{equation} \label{mmbound2}
   R_m(K,\sigma)_q \asymp 
   \left(m^{-\frac{s}{d} + \left(\frac{1}{p} -\frac{1}{q}\right)_+}+ \min\left\{1,\left[\frac{\sigma^2}{m}\right]^{\frac{s}{2s + d}}\right\}\right),\quad  \sigma\geq 0 \quad {\rm and} \quad   m\ge 1. 
   \end{equation}
\end{theorem}

\begin{remark}
    Although we have presented our algorithm and results for recovering a function $f$ in the unit ball of the Besov space $B^s_\tau(L_p(\Omega))$, a simple scaling argument extends this to the case where $f$ lies in a scaled unit ball, i.e., $\|f\|_{B^s_\tau(L_p(\Omega))} \leq M$ for some fixed constant $M$. In this case, $\tilde{f} = f/M$ is in the unit ball. Upon rescaling our measurements in the same way, our observations become $\tilde{y}_i := \tilde{f}(x_i) + \eta_i/M$. Thus, the variance in the rescaled measurements becomes $\tilde{\sigma}^2 = \sigma^2/M^2$. After recovering $\tilde{f}$ from $\tilde{y}_i$ and multiplying the resulting estimate by $M$ to recover $f$, we obtain the bound
    \begin{equation}\label{scaled-equation}
        R_m(K,\sigma)_q \asymp 
   \left(Mm^{-\frac{s}{d} + \left(\frac{1}{p} -\frac{1}{q}\right)_+} + \min\left\{M,M^{\frac{d}{2s+d}}\left[\frac{\sigma^2}{m}\right]^{\frac{s}{2s + d}}\right\}\right),
    \end{equation}
    for $\sigma \geq 0$ and $m \geq 1$.
    The same remark holds for the probabilistic bounds in Theorem~\ref{T:mainprob}.
\end{remark}

The main takeaway from Theorems~\ref{T:mainprob}~and~\ref{T:mainexp} is that the NLA minimax rate scales as
\begin{equation}
   m^{-\frac{s}{d} + \left(\frac{1}{p} -\frac{1}{q}\right)_+}+ \min\left\{1,\left[\frac{\sigma^2}{m}\right]^{\frac{s}{2s + d}}\right\},  \label{eq:NLA-rate}
\end{equation}
and that the dependence on both $\sigma$ and $m$ is sharp. This  provides rigorous justification for the intuitive fact that the difficulty  of the estimation problem depends on the noise level. Furthermore, we see that the two terms in \eqref{eq:NLA-rate} are balanced when
\begin{equation}
    \sigma^2 \asymp m^{-\frac{2s}{d} + \frac{2s + d}{s}\left(\frac{1}{p} -\frac{1}{q}\right)_+}.
    \label{eq:critical-noise-level}
\end{equation}
This value of $\sigma$ is the transition value between the \emph{low-noise} and 
\emph{high-noise} regimes. In particular, this critical noise level depends on the number of data $m$ as well as the norm used to measure the error and the parameters of the model class.

Note that our proposed algorithm uses piecewise polynomials and a thresholding procedure. Although the use of piecewise polynomials for nonparametric function estimation is not new (see, e.g.,~\cite{chaudhuri1994piecewise,kolaczyk2005multiscale,binev2005universal,binev2007universal,morkisz2020complexity}), our specific algorithm, to the best of our knowledge,  has not been studied before.  Finally, we  want to mention that there has been  recent work aimed at generalizing classical statistical tools for deriving minimax lower bounds in expectation 
to the setting of high-probability minimax lower bounds~\cite{ma2024high}.

\subsection{Organization of the Paper}

In \S\ref{S:Besov}, we recall the definition of  Besov classes and some of their properties.  
In \S\ref{S:NOR}, we introduce our numerical algorithm. In \S
\ref{S:fd}, we revisit the classical problem of recovery of a finite dimensional vector from noisy coordinate measurements. In \S\ref{S:ubs}, we prove Theorem \ref{T:mainprob}, in \S\ref{S:lowerbounds}, we prove
Theorem \ref{T:mainexplb}, and
in \S\ref{SS:pTexp} we prove 
 Theorem \ref{T:mainexp}.
Finally, in \S\ref{S:concludingremarks}, we give our concluding remarks.

\section{Besov Spaces}
\label{S:Besov}
In this section, we recall the definition of Besov spaces and  their properties that we will need going forward.  We begin by  assuming that $\Omega\subset \R^d$ is any bounded Lipschitz domain, i.e.  $\Omega$ is an open simply connected bounded set with a boundary $\partial\Omega$ that can be locally described as the graph of  Lipshitz functions (see \cite{DS}).  
Later, we specialize to the case $\Omega=(0,1)^d$.

 If $s>0$, and $0<p,\tau\le\infty$, then the Besov space $B_\tau^s(L_p(\Omega))$ is a space of functions with smoothness of order $s$ in $L_p(\Omega)$.  Here $\tau $
is a secondary index that gives a finer gradation of these spaces.   Frequently, the case $0<p<1$ is avoided in the literature for technical reasons.  However, we want to include this case in what follows since it is an important model class for nonlinear methods of approximation.  
The material in this section is taken for the most part from the papers \cite{DP,DS,DSmono,BDPS}.  The presentation and notation used in this paper is most closely connected to that in  \cite{BDPS}.
We give a condensed version of this material on Besov spaces and the reader will have to refer
to those papers for some of the definitions and proofs.  
The univariate
case is covered in the book \cite{DL}.

If $r$ is a positive integer and $0<p\le \infty $, we define the modulus of smoothness by
\be 
\label{moduli}
\omega_r(f,t)_p:= \sup_{|h|\le t} \|\Delta_h^r(f,\cdot)\|_{L_p(\Omega_{rh})}, \quad t>0,
\ee
where 
\be 
\label{rdiff}
\Delta_h^r(f,\cdot):=(-1)^r\sum_{k=0}^r (-1)^k\binom{r}{k} f(\cdot+kh),
\ee 
is the $r$-th difference of $f$ for $h\in\R^d$ and 
\be 
\label{defOh}
\Omega_h:=\{x\in\Omega: [x,x+h]\subset \Omega\}. 
\ee 
If $s>0$, $0<\tau<\infty,$ and $0<p\le \infty$, then
$B_q^s(L_p(\Omega))$ is defined as the set of all functions in $L_p(\Omega)$ for which
\be 
\label{Bsemi}
|f|_{B_\tau^s(L_p(\Omega))}:= \left [\int_0^1 [t^{-s}\omega_r(f,t)_p]^\tau\frac{dt}{t}\right ]^{1/\tau}<\infty,
\ee 
 where $r$ is the smallest integer strictly bigger than $s$. When $\tau=\infty$, we replace the integral by a supremum in the definition. This is a (quasi-)semi-norm and we obtain the (quasi-)norm for $B_\tau^s(L_p(\Omega))$
by adding $\|f\|_{L_p(\Omega)}$ to this (quasi-)semi-norm. An equivalent (quasi-)semi-norm is given by
\be 
\label{Bsemi1}
  \left [\sum_{k=0}^\infty [2^{ks}\omega_r(f,2^{-k})_p]^\tau \right ]^{1/\tau}.
\ee 
This equivalence is proved by discretizing the integral in \eref{Bsemi} and using the monotonicity of $\omega_r$
as a function of $t$.

Let us remark on the role of $\tau$ in this definition.  If $\tau_1>\tau_2$, then we have that the spaces $B_{\tau_2}^s(L_p(\Omega))\subset B_{\tau_1}^s(L_p(\Omega)) $.  In other words, these spaces get smaller as $\tau$ gets smaller, and   thus all of these spaces are contained in $B_\infty^s(L_p(\Omega))$ once $s$ and $p$ are fixed. The effect of $\tau$ in the definition of the Besov spaces is subtle.  In this paper, the space $B_\infty^s(L_p(\Omega))$ will be the most important case when proving upper estimates for optimal recovery or minimax  because these estimates do not depend on $\tau$.  Accordingly, we use the abbreviated
notation
$B_p^s:=B_\infty^s(L_p(\Omega))$
going forward.

There are many equivalent descriptions of Besov classes.  Most notably, when $B_p^s$ embeds into $L_1(\Omega)$,
one can characterize the elements $f\in B_p^s$ as having a wavelet basis decomposition whose wavelet coefficients belong to a certain sequence
space~\cite{CDD}. We want to avoid the complexity of defining wavelets on domains and therefore will, instead, use a description of Besov spaces via piecewise polynomial approximation.
From the definition of Besov spaces, it follows that a function $f$ is in $B_p^s$ if and only if
\be 
\label{member}
\omega_r(f,t)_{L_p(\Omega)}\le Mt^s,\quad t>0,
\ee
and the smallest $M$ for which \eref{member} is valid is the (quasi-)semi-norm $|f|_{B_p^s}$.
We obtain the (quasi-)norm for $ B_p^s$ by adding the $L_p(\Omega)$ (quasi-)norm to this semi-norm.  When $s$ is not an integer, this space is commonly referred to as the (generalized)  Lipschitz smoothness space of order $s$ in $L_p$. It is important to note that when $s$ is an integer, we take $r=s+1$ in its definition,  and therefore the Besov space $B_p^s$ is not a Lipschitz space 
 in this case.
In view of \eref{Bsemi1}, we have that
 a function $f$ is in $B_p^s$ if and only if
\be 
\label{member1}
\omega_r(f,2^{-k})_{L_p(\Omega)}\le M'2^{-ks},\quad k=0,1,\dots,
\ee
and the smallest $M'$ for which this is true is an equivalent (quasi-)semi-norm.
We proceed now to show how $B_p^s$ is
characterised by local polynomial approximation.

\subsection{Local Approximation by Polynomials}
\label{SS:polyapprox}
Let $\cP_r$ denote the space of algebraic polynomials of order $r$ (total degree $r-1$) in $d$ variables, namely, $\cP_r$ is the span of the monomials $x^k=x_1^{k_1}\cdots x_d^{k_d}$ with the $k_j$'s being  non negative integers which satisfy  $\sum_{j=1}^d k_j<r$.  Note that $\omega_r(P,t)_{L_p(\Omega)}=0$, $t\ge 0$, for all $P\in\cP_r$.   We use the notation
\be 
\label{dimPr}
\rho:= \rho(r):=\dim(\cP_r) = \binom{d+r-1}{d}.
\ee

If $I$ is any cube contained in $\Omega$ and $g\in L_p(I)$, $0<p\le \infty$,  we let
\be 
\label{Er}
E_r(g,I)_p:=\inf_{P\in\cP_r}\|g-P\|_{L_p(I)},
\ee 
denote the error of the $L_p$-approximation of $g$ on $I$  by the elements of $\cP_r$.  A well known result in approximation theory, commonly referred to as Whitney's theorem \cite{storozhenko1978jackson}, says
that for any $g\in L_p(I)$ with $I$ a cube with sidelength $\ell_I$, we have
\be 
\label{WT}
cE_r(g,I)_p\le \omega_r(g,\ell_I)_{L_p(I)}\le CE_r(g,I)_p,
\ee 
with the constants $c,C>0$ depending only on $r,d$ and $p_0$ when $p_0\le p\le\infty$.  Whitney's theorem usually only refers to the lower inequality in \eref{WT}.  However, the upper inequality follows trivially since 
\begin{equation}
\omega_r(g,\ell_I)_{L_p(I)}=\omega_r(g-P,\ell_I)_{L_p(I)}\le C\|g-P\|_{L_p(I)},
\end{equation}
holds for any polynomial $P\in\cP_r$.

If $I\subset\Omega$ is a cube, we  say that $Q\in\cP_r$ is a near best $L_p(I)$-approximation to $g$ with constant $c_0 \ge 1$ if \looseness=-1
\be 
\label{nearbest1}
\|g-Q\|_{L_p(I)}\le c_0 E_r(g,I)_p.
\ee 
It is shown in Lemma 3.2 of \cite{DP} that if $Q\in\cP_r$ is a near best  approximation in $L_{ p}(I)$,  then it is also a near best  approximation in $L_{\bar p}(I)$ whenever $\bar p\ge p$.
  Another important remark is that any near best approximation $Q$ in $L_p(I)$ is near best on any larger cube $J$ which contains $I$ in the sense that
\be 
\label{nearbest2}
\|g-Q\|_{L_{ p}(J)}\le Cc_0 E_r(g,J)_{ p},
\ee 
where now $C$ depends additionally on $|J|/|I|$,  where $|I|$ denotes the measure of $I$.  In other words, a near best $L_p(I)$
approximation is also near best on larger cubes $J$ containing $I$  and larger values $\bar p\ge p$.
We shall use these facts throughout this paper.

\subsection{Polynomial Norms}
\label{SS:polynorms}
All norms on the finite dimensional space $\cP_r$ are equivalent.  In what follows, we need good bounds on the constants that appear when comparing norms.   We introduce the normalized $L_p$ (quasi-)norms
\be
\label{p*norm}
\|g\|^*_{L_p(I)}:= |I|^{-1/p}\|g\|_{L_p(I)},
\ee 
for any cube $I\subset \R^d$ and any $g\in L_p(I)$, $0<p\le \infty$. 
We recall the equivalences given in (3.2) of
\cite{DS} and conclude that for any polynomial $P\in\cP_r$ and any $p,q$, we have
\be 
\label{eqPrnorm}
 \|P\|^*_{L_q(I)}\le C \|P\|^*_{L_p(I)},
 \ee 
 with absolute constants,  provided $r$ is fixed,  and $q_0\le q,p$ with $q_0$ fixed.

\subsection{Besov Spaces and Piecewise-Polynomial Approximation}
\label{Bpwp}
 
We recall how membership in $ B_p^s(\Omega)=B_p^s(\bar \Omega)$, $\Omega=(0,1)^d$, is characterized by piecewise polynomial approximation.  
  For each $k\ge 0$, we define $\cD_k$ to be the set of all  dyadic cubes $I\subset \Omega$  of side length $2^{-k}$.  For specificity,  we take these cubes to be half open, i.e., each $I$ has the form $I=[a_1,b_1)\times \dots \times [a_d,b_d)$.  These cubes give a partition of $[0,1)^d$.  We then define 
$\cS_k=\cS_k(r)$ to be the space of all piecewise polynomials of order $r$ that are subordinate to the partition $\cD_k$. More precisely, a  function $S$ is in $\cS_k(r)$ if and only if   for  each $I\in\cD_k$, we have $S=Q_I$ with $Q_I\in\cP_r$.   Each $S$ has a natural extension to $[0,1]^d$.

 Let $0<p\le \infty$ and $s>0$ and $r>s$. A  function 
 $f\in B_p^s (\Omega)$,   if and only if 
 (see~\cite[Lemma~12.1]{BDPS})
 \be 
   \label{Tbesovpwpchar}
   \dist(f,\cS_k(r))_{L_p(\Omega)}\le M2^{-ks},\quad k\ge 0.
   \ee 
   The smallest $M$ for which \eref{Tbesovpwpchar} holds can be taken as a (quasi-)semi-norm for $B_p^s(\Omega)$.

  \subsection{Least-Squares Approximation}
  \label{SS:ls}
  In this section, we construct and analyze a projection onto the space $\cS_k$ which is based upon performing a least squares fit to the function values on each cube $I\in \cD_k$.
Let $k \leq n-r$ be fixed and $I\in \cD_k$. Observe that $I$ is a dyadic cube with side length $2^{-k}$ and so the intersection $I\cap G_n$ of $I$ with the observation grid (see \eref{Gn}) consists of a regular grid with $N_k = 2^{n-k}$ points in each direction. We remark that since $k \leq n-r$, we have $N_k^d > \rho$. 
  
  Since the least squares  procedure is invariant to scaling, we begin with the case  $I = [0,1)^d$ and
 the set of points $\{z_i, i=1,\dots,N^{d}\}$, from the
grid $\Lambda =\{0,1/N,\ldots,1-1/N\}^d$, with some $N^d>\rho$.  
This is the  set of these data sites for performing a least squares fitting
   to a given function $f$ by polynomials from $\cP_r$. We   
    introduce the measure
   \be 
   \label{Dirac}
   \mu:=\frac{1}{N^{d}}\sum_{i=1}^{N^{d}} \delta_{z_i}=\frac{1}{|\Lambda|}\sum_{z_i\in\Lambda} \delta_{z_i},
   \ee 
   where the $\delta_{z_i}$ are Dirac measures 
   at $z_i$. We 
   then have the Hilbert space $L_2(\mu)$, which is identified with functions defined on the data sites $\Lambda$, with the inner product
   \be 
   \label{ip1}
   \langle f,g\rangle:=\frac{1}{|\Lambda|}\sum_{z_i\in \Lambda}f(z_i)g(z_i).
   \ee 
   
   Recall that $\rho=\dim(\cP_r)$  and let $Q_1,\dots,Q_\rho$, be an orthonormal system for $\cP_r$ viewed as a subspace of $L_2(\mu)$. Since $N^d = |\Lambda| > \rho$, for any function $f$ defined on the data sites $\Lambda$, the polynomial
   \be 
   \label{lsP}
   Pf:=\sum_{j=1}^\rho\langle f,Q_j\rangle Q_j,
   \ee 
   is the least squares approximation to $f$ from $\cP_r$. We will need the following technical lemma about the orthonormal basis $Q_1,...,Q_\rho$, which is proved in Appendix~\ref{Lemma22}.
   \begin{lemma}
    \label{L:polynorms}
    Let $I$ be a bounded cube in $\R^d$ and let $\rho:=\dim(\cP_r)$.  Let $\Lambda_I:=\{z_j, \,j=1,\ldots,N^d\}$ be a tensor product grid of $N^d$ points from $I$ with spacing $1/N$,  and   let the $Q_{I,j}$, $j=1,\dots,\rho$, be  an orthonormal basis for $\cP_r$ on $I$, considered as a subspace of the Hilbert space $L_2(\mu_I)$.
    If $N^d>\rho$ and $q_0\le q\le \infty$, then the following (quasi-)norms of a polynomial $Q=\sum_{j=1}^\rho \beta_jQ_{I,j}$ are equivalent  with constants of equivalency depending only on $r,d$ and $q_0>0$ but not depending on $N$ or $q$:
    
    \noindent
    {\rm (i)} \quad $\|Q\|_{L_q(I)}^*$;  
    
    \noindent
    {\rm (ii)} \quad $\|(Q(z_j))\|^*_{\ell_q(\Lambda_I)}:=N^{-d/q}\|(Q(z_j))_{z_j\in\Lambda_I}\|_{\ell_q(\Lambda_I)}$; 
    
    \noindent
    {\rm (iii)} \quad $\|(\beta_j)_{j=1}^\rho\|_{\ell_q}$. 
 
In particular, it follows that 
\be 
\label{QIjbound}
\|Q_{I,j}\|_{L_\infty(I)}\le C, \quad I\in\cD_k,\ j=1,\dots,\rho,
\ee 
where $C$ depends only on $r,d$.
\end{lemma}

It follows from the bound \eqref{QIjbound} that for a function $f\in C(I)\subset L_2(\mu)$ the linear mapping $f\mapsto Pf$ is a projection onto $\cP_r$ and satisfies
   \be 
\label{Pbounded}
\|Pf\|_{C(I)}\le C\|f\|_{C(I)}
\ee 
for a constant $C$ depending only upon $r$ and $d$, and thus we have
\be 
\label{nb1}
\|f-Pf\|_{C(I)} \le (1+C)\dist(f,\cP_r)_{C(I)},\quad f\in C(I).
\ee 
   
 Let us now consider any cube $I\in \cD_k$.  We rescale the above construction to $I$ and denote by $P_If$ the resulting least squares polynomial fit from $\cP_r$ to this data on $I$, i.e., \eref{lsP} rescaled to $I$. 
 We then define
 \be 
 \label{defSkLS}
 S_kf:=\sum_{I\in \cD_k} (P_If) \chi_I.
 \ee

 \begin{lemma}
     \label{L:LSpwp}
     Let $0<p\le \infty$ and $s>d/p$.  If $f\in B_p^s=B_p^s(\Omega)$, then for each $0 < q\le\infty$, the piecewise polynomial $S_kf$ in \eref{defSkLS} satisfies
 \be 
 \label{Lqapprox1}
 \|f-S_kf\|_{L_q(\Omega)}\le C|f|_{B_p^s(\Omega)} 2^{-k(s-(d/p-d/q)_+) },    \quad \quad k\le n-r,
 \ee
 where $C$ depends only on $p,q,s,d,r$.
 \end{lemma}

 \begin{proof}
 The proof is similar to the proof of Theorem 2.1 in \cite{BDPS}, which applies to interpolating polynomials instead of the least squares polynomial fit, and we do not repeat it in detail here. Indeed, the only property of the interpolating operator which was used in \cite{BDPS} was boundedness in $C(I)$, which the operator $S_k$ also has, see  \eref{Pbounded}.
 \end{proof}

\subsection{Multiscale Decompositions}
\label{SS:mspp}
Let $\Omega=(0,1)^d$ and let us fix $f\in B_p^s(\Omega)$ with $s>d/p$.  Recall that each such function is uniformly continuous and hence has a natural extension to $\overline \Omega$.  We assume that the data sites are the tensor product grid $G_n$ \eref{Gn},  $n\ge 1$, and therefore $m=|G_n|=2^{nd}$.  Note that  the number of data sites in any dyadic cube $I\in \cD_k$ is $2^{(n-k)d}$ for $0\le k\le n$.  We recall that $\rho:=\dim(\cP_r)$. We need to ensure that each $I\in\cD_k$ contains more than $\rho$ data points for every level $k$. This can be done by restricting $k$ to be a constant less than $n$. For simplicity, we will restrict $k\le n-r$, which one can easily verify that is enough. It follows that the least squares operator $P_I$ from the previous subsection is well defined
and is represented by \eref{lsP}. We let 
\be 
\label{defS}
S_k:=S_kf:= \sum_{I\in \cD_{k}} (P_If)\chi_I \quad {\rm and}\quad  \widetilde S_k:=\sum_{I\in \cD_{k}} (P_Iy)\chi_I,\quad k\le n-r,
\ee 
where $y$ is the vector of noisy data observations $(y_i)_{ x_i\in \Lambda_I}$ from \eqref{eq:observations} (which can be viewed as an element of $L_2(\mu)$). 

If we define the piecewise polynomials
\begin{equation}
T_0:=S_0, \ \widetilde T_0:=\widetilde S_0, \quad   T_k:=S_k-S_{k-1}, \quad \widetilde T_k:=\widetilde S_k-\widetilde S_{k-1}, \quad 1\le k\le n-r, 
\end{equation}
then we have the following multiscale  representations

\be 
\label{repf} 
S_k=\sum_{j=0}^k T_j \quad {\rm and}\quad \widetilde S_k=\sum_{j=0}^k \widetilde T_j, \quad 0\le k\le n-r,
\ee
which we shall exploit in our recovery algorithms.  
The following lemma holds.
\begin{lemma}
\label{MainL}
For every $0\le k\le n-r$ we have the representation
\be
\label{fin}
 T_k:=\sum_{I\in\cD_k} \left[\sum_{j=1}^\rho  c_{I,j}Q_{I,j}\right]\chi_I,
\ee
where the vector 
 $\nu_k$ of coefficients defined by
\be 
\label{defnuk}
\nu_k:=(c_{I,j}),\quad I\in \cD_k,\quad j=1,\ldots,\rho, 
\ee
has length $L_k:=\rho 2^{kd}$, and its $\ell_p$ norm satisfies
\be 
 \label{know10}
 \|\nu_k\|_{\ell_p}\le C2^{-k(s-d/p)},\quad 0\le k\le n-r,
 \ee
with $C$ depending only upon $s$, $p$, $r$ and $d$. 
\end{lemma}
\begin{proof}
We have  
\be 
\label{twopwp}
T_k=S_k-S_{k-1}=\sum_{I\in\cD_k}R_I\chi_I,\quad 1\le k\le n-r,\quad \hbox{with}
\quad 
R_I:=P_I-P_{I'},\quad   I \in \cD_k,
\ee
where for each $I\in \cD_k$, we denote by $I'\in \cD_{k-1}$ the parent of $I$, i.e., the cube in $\cD_{k-1}$ that contains $I$. Note that $T_k$ clearly lies in the space $\cS_k$ of piecewise polynomials on the grid $\cD_k$. It follows from Lemma \ref{L:LSpwp} with $q = p$ that
\be
    \|T_k\|_{L_p(\Omega)} \leq C|f|_{B_p^s(\Omega)} 2^{-ks}.
\ee
Now we use Lemma \ref{L:polynorms} to see that
\begin{equation}
    \|\nu_k\|_{\ell_p}^p = \sum_{I\in \cD_k} \sum_{j=1}^\rho |c_{I,j}|^p \leq C\sum_{I\in \cD_k}(\|T_k\|_{L_p(I)}^*)^p = C2^{kd} \|T_k\|_{L_p(\Omega)}^p,
\end{equation}
from which \eqref{know10} follows.
\end{proof}

 Similarly, one can show  the following lemma for $\widetilde S_{n-r}$.
\begin{lemma}
\label{MainLtilde}
We have the representation
\be
\label{fintilde}
\widetilde T_k:=\sum_{I\in\cD_k} \left[\sum_{j=1}^\rho  c^*_{I,j}Q_{I,j}\right]\chi_I,
\ee
where 
\begin{equation}
c^*_{I,j}:=c_{I,j}+\eta_{I,j}, \quad I\in \cD_k,\ 0\le k\le n-r,\ j=1,\dots,\rho, 
\end{equation}
are the noisy observation of the true  $c_{I,j}$'s,
polluted by the additive  Gaussian $\cN(0,\sigma_{I,j}^2)$ noise $\eta_{I,j}$ with variance
\be 
\label{varianceIj}
0\le \sigma_{I,j}^2\le C_\rho 2^{-(n-k)d}\sigma^2, \quad I\in \cD_k,\ j=1,\dots,\rho.
\ee 
\end{lemma}
\begin{proof}
Using the same notation and analysis as in Lemma \ref{MainL} we have
\be 
\label{twopwptilde}
\widetilde T_k=\widetilde S_k-\widetilde S_{k-1}=\sum_{I\in\cD_k}\widetilde R_I\chi_I, \quad 1\le k\le n-r,
\quad \hbox{with}\quad 
  \widetilde R_I:=\widetilde P_I-\widetilde P_{I'}, \quad I \in \cD_k,
\ee
where for each $I\in \cD_k$, 
\be 
\label{recall1}
\widetilde P_I= \frac{1}{\#(\Lambda_I)}\sum_{j=1}^\rho \left[\sum_{ x_i \in \Lambda_I} (f( x_i)+\eta_i)Q_{I,j}( x_i)\right ]Q_{I,j}=P_I+\sum_{j=1}^\rho   \tilde\eta_{I,j}Q_{I,j}.
\ee 
Here
\be 
\label{PIrep1}
 \widetilde\eta_{I,j}:= \frac{1}{\#(\Lambda_I)}   \sum_{ x_i \in \Lambda_I} Q_{I,j}(  x_i) \eta_i, \quad I\in\cD_k, \ j=1\dots,\rho,
\ee 
are random variables.  
 This gives
 \be 
 \label{repPRItilde}
\widetilde R_I=  
 \sum_{j=1}^\rho c_{I,j} Q_{I,j}+   \sum_{j=1}^\rho \eta_{I,j}Q_{I,j}=\sum_{j=1}^\rho c^*_{I,j} Q_{I,j},
 \ee 
where
\begin{equation}
 \eta_{I,j}:=\widetilde \eta_{I,j} -  \sum_{i=1}^\rho\widetilde\eta_{I',i}\langle  Q_{I',i},Q_{I,j}\rangle_I.
\end{equation}
Therefore \eref{fintilde} follows from \eref{repPRItilde} and \eref{twopwptilde}.

   Using \eref{PIrep1} and Lemma \ref{L:polynorms}, we  find that
\begin{equation}
  \eta_{I,j}=\sum_{i:\,x_i\in\Lambda_{I'}}  \beta_{i,j}\eta_i,
\quad \hbox
{where}
\quad 
 |\beta_{i,j}|\le C_\rho \#(\Lambda_I)^{-1}, \quad I\in\cD_k, \ j=1,\dots,\rho.
\end{equation}
   It follows that each of the $\eta_{I,j}$ is a Gaussian $\cN(0,\sigma_{I,j}^2)$ with variance
\begin{equation}
0\le \sigma_{I,j}^2\le C_\rho 2^{-(n-k)d}\sigma^2, \quad I\in \cD_k,\ j=1,\dots,\rho,
\end{equation}
because $\#(\Lambda_I)=2^{(n-k)d}$. 
\end{proof}

\section {A Numerical Recovery Algorithm}
\label{S:NOR}
In this section, we fix $m,\sigma,s,p,q$, and $d$ and describe our algorithm $A$ for approximating $f\in  B_p^s(\Omega)$ from the noisy data $y$, see \eref{eq:observations}. We do this under the assumption that the data sites $\cX=G_n$, that is, they  lie on the tensor product grid described in \eref{Gn} consisting of  $m=2^{nd}$ points.  Our algorithm creates  a function $\hat f:=A(y)$ that approximates $f$ in the $L_q(\Omega)$ norm with $q\in [1,\infty)$. 

 {\bf Step 0:}  If $\sigma^2\ge m$, we define $A(y):=0$ for all $y$.  If $\sigma^2<m$, we proceed to the following two steps.

{\bf Step 1:} From the noisy observations $\{y_i\}$, we construct the polynomials $\widetilde P_I$, $I\in\cD_k$, $k\le n-r$, using least squares approximation from $\cP_r$. 
 This enables us to compute the polynomials $\widetilde R_I$, see \eref{repPRItilde}, for $I\in \cD_k$,  $0\le k\le n-r$,
 \be 
 \label{RItilde}
 \widetilde R_I=\widetilde P_I-\widetilde P_{I'}= \sum_{j=1}^\rho c_{I,j}^*Q_{I,j}, \quad \hbox{where}\quad c^*_{I,j}:= c_{I,j}+\eta_{I,j},
 \ee 
and calculate the coefficients    $\{c_{I,j}^*\}$. Recall that, see Lemma \ref{MainLtilde},
 \be 
\label{deftildeTk}
\widetilde T_k=\sum_{I\in\cD_k} \widetilde R_I\chi_I=\sum_{I\in\cD_k} \left[\sum_{j=1}^\rho c^*_{I,j}Q_{I,j}\right]\chi_I.
\ee

{\bf Step 2:}  For each  dyadic level $k$  and each index pair $(I,j)$ with $I\in\cD_k$, we shall use
the, calculated in Step 1 coefficients,  $\{c_{I,j}^*\}$ to derive new  coefficients $\hat c_{I,j}$, $I\in \cD_k$, $j=1,\dots,\rho$, which approximate the true coefficients $c_{I,j}$ at this dyadic level. The method for calculating the $\hat{c}_{I,j}$'s from the noisy observations $c_{I,j}^*$ (which will be hard-thresholding)  will be described and analyzed in  \S\ref{S:fd}. This part of the algorithm requires the a priori knowledge of the noise level $\sigma$. Notice that the $\hat c_{I,j}$'s will be random variables. We then define the piecewise polynomial function
\be  
\label{defhatTk}
\hat T_k=\sum_{I\in\cD_k} \left[\sum_{j=1}^\rho \hat c_{I,j}Q_{I,j}\right]\chi_I,
\ee 
and consider
$\hat f:=A(y):=\sum_{k=0}^{n-r} \hat T_k$,
to be  our approximation to $f$.

\subsection{Preliminary Analysis}
To analyze the approximation error in $L_q$ of our algorithm, we use Lemma \ref{L:LSpwp} and Lemma \ref{MainL} to obtain the estimate.
\begin{eqnarray}
\label{bound1}
\|f-\hat f\|_{L_q(\Omega)}&\le &  \|f-S_{n-r}\|_{L_q(\Omega)}  +\sum_{k=0}^{n-r}\|T_k-\hat T_k\|_{L_q(\Omega)}
\nonumber \\
&\le&  C2^{-(n-r)(s-(d/p-d/q)_+)} +\sum_{k=0}^{n-r}\|T_k-\hat T_k\|_{L_q(\Omega)}\nonumber\\
&\le&  Cm^{-(s/d-(1/p-1/q)_+)} +\sum_{k=0}^{n-r}\|T_k-\hat T_k\|_{L_q(\Omega)}.
\end{eqnarray} 

In order to bound $\|T_k-\hat T_k\|_{L_q(\Omega)}$, we will use the inequality
\be 
\label{kerror} 
\|T_k-\hat T_k\|^q_{L_q(\Omega)}\le C\sum_{I\in\cD_k}\left[\sum_{j=1}^\rho | c_{I,j}-\hat c_{I,j}|^q\right] |I|=C2^{-kd}\sum_{I\in\cD_k}\left[\sum_{j=1}^\rho | c_{I,j}-\hat c_{I,j}|^q\right].
\ee
Next, we organize our notation as follows. Recall that $\nu_k$ was defined in \eqref{defnuk} as the vector of coefficients $c_{I,j}$, $I\in\cD_k$, $j=1,\dots,\rho$, in the representation of $T_k$ at level $k$. In going forward, we use the notation 
\begin{equation}
\nu^*_k:=(c^*_{I,j}),\quad \hbox{and}\quad 
\eta^*_k:=(\eta_{I,j}),
\end{equation}
so that   
\begin{equation}
\nu_k^* = \nu_k + \eta^*_k. 
\end{equation}
The vector 
 $\hat \nu_k$ is our approximation (which we have not yet defined)  to $\nu_k$ obtained from the noisy observation $\nu_k^*$. 
For a finite vector, say  $v=(v_1, \ldots,v_L)\in \mathbb{R}^L$, we introduce  the weighted  $\ell_q$-norm, given by
\be
   \|v\|^*_q := \left(\frac{1}{L}\sum_{i=1}^L |v_i|^q\right)^{1/q}.
\ee
Utilizing that $\nu_k$ and $\hat \nu_k$  are vectors of length $L_k = \rho2^{kd}$, we rewrite the bound \eqref{kerror} as 
\be
\label{q}
    \|T_k-\hat T_k\|_{L_q(\Omega)}\le C\|\nu_k - \hat{\nu}_k\|^*_{q}.
\ee
We want to define $\hat\nu_k$ to make \eref{q} small.
Putting \eref{q} together with \eqref{bound1} results in  
\be\label{f-error}
\|f-\hat f\|_{L_q(\Omega)}\le C\left(m^{-\frac{s}{d} +(\frac{1}{p}-\frac{1}{q})_+} + \sum_{k=0}^{n-r}\|\nu_k - \hat{\nu}_k\|^*_{q}\right).
\ee
In the next sections, we  explain how we obtain $\hat\nu_k$ by   hard-thresholding the entries in $\nu_k^*$  and bound the resulting error $\|\nu_k - \hat{\nu}_k\|^*_{q}$ in probability.

\section{Finite-Dimensional Estimates}
\label{S:fd}
Here, we study the problem of recovering a vector $v\in \mathbb{R}^L$ from noisy observations $v^* = v + \xi$ with error measured in the norm $\|\cdot\|^*_q$ with $1\leq q<\infty$. We assume that each entry $\xi_i$ of the noise vector $\xi$ is a mean zero Gaussian with variance $\sigma^2_i$ bounded by $\tilde\sigma^2 \geq  0$. In this analysis, we allow the noise in the different components to potentially be correlated.   We will later use the results in this section to bound the error terms $\|\nu_k - \hat{\nu}_k\|^*_{q}$ in \eqref{f-error}. 

The problem of recovering $v$ has been studied before in the literature (see for instance \cite{donoho1990minimax,donoho1994minimax}). One popular method for estimating $v$ is hard-thresholding, defined by
\be\label{hard-thresholding-1092}
    \hat v_i = \mathrm{thresh}_\lambda(v^*_i) := \begin{cases}
                v^*_i, & |v^*_i| > \lambda,\\
                0, & |v^*_i| \leq \lambda,
              \end{cases}
\ee
where $\lambda \geq 0$ is a parameter to be chosen. The key property of hard-thresholding that we are going to use is stated in  the next  simple Lemma (see Lemma 2 in \cite{DeJu}). As remarked in \cite{DeJu,donoho1994ideal,donoho1994minimax,donoho1990minimax}, any coordinatewise estimation rule satisfying the bound of the Lemma, including the popular soft-thresholding rule, can be used to obtain similar guarantees.

\begin{lemma}\label{threshold-error-lemma}
    Let $x,\,\epsilon\in \mathbb{R}$. Then for any $\lambda \geq 0$,  we have the estimate
    \begin{equation}\label{basic-thresholding-bound}
        |\mathrm{thresh}_\lambda(x + \epsilon) - x| \leq 3(\min\{|x|,\lambda\} + |\mathrm{thresh}_{\lambda/2}(\epsilon)|).        
    \end{equation}
\end{lemma}
\begin{proof}
    Indeed, if $|x| \leq \lambda / 2$ and $|\epsilon| < \lambda / 2$, then
    \begin{equation}
        \mathrm{thresh}_\lambda(x + \epsilon) = 0\Rightarrow |\mathrm{thresh}_\lambda(x + \epsilon) - x| = |x| \leq \min\{|x|,\lambda\}.
    \end{equation}
    Otherwise, we use that for all $z\in \mathbb{R}$, we have $|\mathrm{thresh}_\lambda(z) - (z)| \leq \lambda$, which implies that
    \begin{equation}
        |\mathrm{thresh}_\lambda(x + \epsilon) - x| \leq \lambda + |\epsilon|.
    \end{equation}
    Now, if $|\epsilon| \geq \lambda/2$, it follows that $\lambda \leq 2|\epsilon|$ so that
    \begin{equation}
        |\mathrm{thresh}_\lambda(x + \epsilon) - x| \leq 3|\epsilon| = 3\cdot|\mathrm{thresh}_{\lambda/2}(\epsilon)|.
    \end{equation}
    On the other hand, if $|\epsilon| < \lambda / 2$ and $|x| > \lambda/2$, then $\lambda / 2 < \min\{|x|,\lambda\}$, and we have
    \begin{equation}
        |\mathrm{thresh}_\lambda(x + \epsilon) - x| \leq \lambda + |\epsilon| \leq (3/2)\lambda \leq 3\min\{|x|,\lambda\}.
    \end{equation}
    So, in all cases, we  have the estimate \eqref{basic-thresholding-bound}.
    \end{proof}
We use  this Lemma to show that the error $\|v - \hat v\|^*_{q}$ can be bounded by a deterministic and stochastic component.  We first define a new random variable $\xi_\lambda$, whose components are obtained from the  components of a Gaussian random variable $\xi$ via 
hard thresholding,
\be\label{definition-of-eta-lambda}
    \left(\xi_{\lambda}\right)_i :=\mathrm{thresh}_{\lambda/2}(\xi_i)= \begin{cases}
        \xi_i, & |\xi_i| \geq \lambda / 2,\\
        0, & |\xi_i| < \lambda / 2,
    \end{cases}
\ee
and prove  the following theorem.

\begin{theorem}
\label{T;detstoch}
Let  $v$ be a vector from $ \R^L$ for which we have the  noisy observations $v^* = v+ \xi$.  For $\lambda\geq 0$, we  define $\hat v$ by \eref {hard-thresholding-1092}.  Then, for any $1\leq q<\infty$,   and $0<p\leq q$, we have
\begin{equation}\label{Tds}
    \|v - \hat v\|^*_{q} \leq   3   \left [\|\min\{|v|,\lambda\}\|^*_{q}+\|\xi_\lambda\|^*_{q}\right ]\leq 3\left [[\|v\|_p^*]^{p/q}\lambda^{1-p/q}+\|\xi_\lambda\|^*_{q}\right],
\end{equation}
where
 $\min\{|v|,\lambda\}$ is defined componentwise.

\end{theorem}
\begin{proof}
We apply Lemma \ref{threshold-error-lemma} to each component of $v$ to obtain
\begin{equation}
| v_i-\hat v_i|\le 3(\min\{|v_i|,\lambda\} + |(\xi_\lambda)_i|).
\end{equation}
We now apply the $\|\cdot\|_q^*$ to derive the first inequality in \eref{Tds}.
Next, for each component $v_i$ of $v$, we have
    \begin{equation}
    \min\{|v_i|,\lambda\}\le  |v_i|^{p/q}\lambda^{1-p/q}.
    \end{equation}
Applying   $\| \cdot\|_q^*$ we arrive at the last inequality in \eref{Tds}.
\end{proof}

Notice that the first term on the right side of \eref{Tds} is deterministic since it does not depend on the noise draw.  We next want to give a bound for the second term which is stochastic. To derive our bound,
we introduce the notation.
\begin{equation}
F:=\frac{1}{L}\sum_{i=1}^L f_i, \quad \hbox{where}\quad     f_i := \left|\xi_{\lambda}\right|^q_i = \begin{cases}
        |\xi_i|^q, & |\xi_i| \geq \lambda / 2,\\
        0, & |\xi_i| < \lambda / 2.
    \end{cases}
\end{equation}  
Then, $\|\xi_\lambda\|^*_q =F^{1/q}$. 

\begin{theorem}
\label{T:etaprob}
For any 
 increasing, convex function $\phi$ on $\R_+:=[0,\infty)$ with $\phi(0)=0$    and any $1\leq q<\infty$,  we have
\be
\label{Tetaprob1}
   \mathbb{P}(\|\xi_\lambda\|_q^*\ge T) \le  \tilde \sigma^{-1}\int_{\lambda / 2}^\infty
\frac{\phi(x^q)}{\phi(T^q)}e^{-x^2/2\tilde\sigma^2}\,dx, \quad T>0.
   \ee 
\end{theorem}
\begin{proof}
First note that from the convexity of $\phi$  and  Jensen's inequality, we have
\begin{equation}
\phi(F)\leq \frac{1}{L}\sum_{j=1}^L \phi(f_i),
\end{equation}
and
\be 
\label{Jensen}
    \mathbb{E}(\phi(F)) \leq \frac{1}{L}\sum_{j=1}^L  \mathbb{E}(\phi(f_j))\le \max_{1\le j\le L} \mathbb{E}(\phi(f_j)).
\ee 
We can therefore estimate
\begin{eqnarray}
\mathbb{P}( \|\xi_\lambda\|_q^*\ge T)&=& \mathbb{P}( F\ge T^q)\leq  \mathbb{P}( \phi(F)\ge \phi(T^q))\le  (\phi(T^q))^{-1}\mathbb{E}(\phi(F)) \nonumber \\
&\le& \frac{1}{\phi(T^q)}\max_{1\le j\le L} \mathbb{E}(\phi(f_j)),
\end{eqnarray}
where we used Markov's inequality, \eref{Jensen}, and the monotonicity of $\phi$.   
   Recall that  $\sigma_j^2$,  the variance of $\xi_j$, is such that $\sigma_j \leq \tilde\sigma$. We then have
\begin{equation}
[\phi(T^q)]^{-1}\mathbb{E}(\phi(f_j)) \le \sigma_j^{-1}\int_{\lambda / 2}^\infty\frac{\phi(x^q)}{\phi(T^q)}e^{-x^2/2\sigma_j^2}\,dx,
\end{equation}
where we change  variables by  setting $x = (\sigma_j/\tilde \sigma) z$, $dx = (\sigma_j/\tilde \sigma)dz$, to get
\begin{equation}
[\phi(T^q)]^{-1}\mathbb{E}(\phi(f_j)) \le \tilde \sigma^{-1}\int_{(\lambda / 2)(\tilde\sigma / \sigma_j)}^\infty\frac{\phi(((\sigma_j/\tilde\sigma) z)^q)}{\phi(T^q)}e^{-z^2/2\tilde\sigma^2}\,dz.
\end{equation}
Finally, we use  that $\sigma_j \leq \tilde\sigma$ and that $\phi$ is increasing and non-negative to obtain
\begin{equation}
[\phi(T^q)]^{-1}\mathbb{E}(\phi(f_j)) \le \tilde \sigma^{-1}\int_{\lambda / 2}^\infty\frac{\phi(z^q)}{\phi(T^q)}e^{-z^2/2\tilde\sigma^2}\,dz,
\end{equation}
 which completes the proof of the theorem.
 \end{proof}

Given the above theorem,  we will obtain the best (smallest) probability  bound for the stochastic term in \eref{Tds}, %
by minimizing over the choice of $\phi$.  To help decide which $\phi$ to employ,
we will use the following elementary result, proved in Appendix~\ref{LemmaK}.
 
\begin{lemma}
\label{L:known}
 For any $a,q>0$, we have 
 \begin{equation}
 \int_a^\infty x^qe^{-x^2/2}\,dx\le C(q)e^{-a^2/4}.
 \end{equation}
\end{lemma}
The following theorem is the main result of this section. 
\begin{theorem}
\label{T:mainthresh1}
If $\xi\in\R^L$ has components each of which is a mean zero Gaussian with variance at most $\tilde \sigma^2$, then for any threshold $\lambda\geq 0$ and any $T>0$, we have 
\be
\label{Tetaprob2}
   \mathbb{P}(\|\xi_\lambda\|_q^*\ge T) \lsim 
    { T^{-q}{\tilde \sigma}^qe^{-
(\tilde\sigma^{-1}b)^2/4} },
\quad \hbox{where}\quad 
b := \max\{2^{-1}\lambda,2^{-1/q}T\},
 \ee 
   and the constant in $\lsim$ depends only on $q$.
\end{theorem}
\begin{proof}
We define $\phi(x):=(x- T^q/2)_+$, for $x\ge 0$.  It follows from Theorem \ref{T:etaprob} that
\begin{eqnarray} 
\label{case21}
 \mathbb{P}(\|\xi_{\lambda}\|_q^*\ge T) & \le &     \tilde\sigma^{-1}\int_{\lambda / 2}^\infty
\frac{\phi(x^q)}{\phi(T^q)}e^{-x^2/2\tilde\sigma^2}\,dx 
=\tilde \sigma^{-1}\int_{b}^\infty
\frac{\phi(x^q)}{\phi(T^q)}e^{-x^2/2 \tilde\sigma^2}\,dx  \nonumber \cr
&=&  [\phi(T^q)]^{-1}\int_{\tilde\sigma^{-1}b} ^\infty \phi(\tilde\sigma^qu^q) e^{-u^2/2}\, du\leq 2 T^{-q}\tilde\sigma^q\int_{\tilde\sigma^{-1}b} ^\infty  u^q e^{-u^2/2}\, du \nonumber\\
&\lsim& 
{ T^{-q}{\tilde \sigma}^qe^{-
(\tilde\sigma^{-1}b)^2/4} },
\end{eqnarray}
where 
we used the fact that $\phi(x)\leq x$ on $[0,\infty)$ and  had applied 
Lemma \ref{L:known}.
\end{proof}

\section{Proofs of  Upper Bounds}
\label{S:ubs}

 In this section, we  show that our numerical algorithm from \S 3 will satisfy the upper bound of the theorems formulated in \S \ref{intro}, provided we properly choose the thresholds $\lambda_k$ at each dyadic level $k$.  Throughout this section, we fix $s,p,q$, so that $s>d/p$, and we assume that we are in the primary case, i.e., \eref{parameters} holds.  

 \subsection{Proof of   Theorem \ref{T:mainprob}}
\label{SS:ub1}
 Let us begin by making two remarks that will allow us to restrict the range of the parameters in what follows.

 \begin{remark} 
 \label{R:p<q}
 We can assume that $p\le q$.
 \end{remark}
    Indeed, when $p\ge q$ we have $U(B_\tau^s(L_p(\Omega)) \subset U(B_\tau^s(L_q(\Omega))$ and the right side of \eref{alpha} is  always equal to $2$. Therefore  this claim follows from the case $p=q$.  We assume that $p\le q$ in going forward in this section.

 \begin{remark} 
 \label{R:sigma}
 We can assume that $\sigma^2 < m $.
 \end{remark}
 Indeed, recall that when $\sigma^2\ge m$, we define $A(y):=0$ for all $y$. Since $f\in U(B^s_\tau(L_p))\subset U(B^s_\infty (L_p))$ and $\|f\|_{L_q(\Omega)}\le C_0\|f\|_{B_\infty^s(L_p(\Omega))}\leq C_0$, the probability in \eref{T1} will
 always be zero if the  constant $C$ in \eref{T1} is chosen such that   $C\geq C_0$.  Hence, \eref{T1} will automatically be satisfied for all $t>0$ whenever $\sigma^2\ge m$.
 We assume $\sigma^2<m$ in going forward in this section.

 \begin{remark} 
 \label{R:t}
 In order  to prove \eref{T1}  for any fixed $q$,  it is enough to consider $t\ge 1$.
 \end{remark}
 Indeed, once \eref{T1} is proven for a given $C,c$ %
  and $t\ge 1$, then it holds for all $t< 1$ by simply adjusting the constant $C$ in \eref{T1}.

  \bigskip
 In order to bound $\|f-\hat f\|_{L_q(\Omega)}$,  we will use \eref{f-error}. So our task is to prove  a bound for
 \be 
 \label{want11}
 E^*:=\sum_{k=0}^{n-r}\|\nu_k-\hat \nu_k\|^*_q,
\ee
where $\nu_k$
are defined in \eref{defnuk} and $\hat \nu_k={\rm thresh}_{\lambda_k}(\nu_k^*)$ are  obtained by thresholding the observed coefficients, i.e., the entries in the vector $\nu_k^*=\nu_k+\eta^*_k$.  Here,  the ${ \rm thresh}_{\lambda_k}$ is performed coordinatewise, with parameters $\lambda_k$ specified later in \eqref{choicelambdak}. Notice that $E^*$ is a random variable depending on our noise draw.

For the remainder of this section, let  $f\in B_\tau^s(L_p(\Omega))$.
Since our performance theorems do not depend on $\tau$, it is enough to consider $\tau=\infty$ which gives the largest model class. We introduce the quantity
\be 
\label{defepsilon}
\e:=\left [\frac{\sigma^2}{m}\right ]^{\frac{s}{2s+d}},
\ee 
which occurs prominently in the bounds we want to prove. Note that $\epsilon < 1$ since $\sigma^2 < m$.
We know that, see Lemma \ref{MainL},
\be\label{l-p-bound-nu}
\|\nu_k\|^*_{p} \leq C2^{-ks},\quad 0\le k\le n-r,
\ee
 since the true coefficient vector $\nu_k$ has length $L_k = C2^{kd}$ and satisfies \eqref{know10}. 
 Each component of the noise $\eta^*_k$ is a Gaussian random variable, see Lemma \ref{MainLtilde}, with variance bounded by (see \eqref{varianceIj})
\be
\label{variance-equation-11571}
\sigma_k^2 :=  C2^{-(n-k)d}\sigma^2 = C\left[\frac{\sigma^2}{m}\right]2^{kd}, \quad 0\le k\le n-r.
\ee
 It follows from Theorem \ref{T;detstoch} applied with $\lambda = \lambda_k$, $v = \nu_k$, and $\xi = \eta_k^*$, 
 that 
 \begin{equation}
 \|\nu_k-\hat \nu_k\|^*_q\leq C\left [
 2^{-ksp/q}\lambda_k^{1-p/q}+\|\eta_{\lambda_k}\|^*_q\right], \quad 0\le k\le n-r, 
 \end{equation}
 where we used the notation $\eta_{\lambda_k}:=\mathrm{thresh}_{\lambda_k/2}(\eta_k^*)$.
Therefore, from \eref{f-error} and \eref{want11}, we obtain
\be
\label{want1}
\|f-\hat f\|_{L_q(\Omega)}\le   C\left[m^{-\frac{s}{d} +(\frac{1}{p}-\frac{1}{q})_+}+E^*\right]
\le   C\left[m^{-\frac{s}{d} +(\frac{1}{p}-\frac{1}{q})_+}+\Sigma_1+\Sigma_2\right],
\ee 
with 
\be 
\label{defsigma1}
\Sigma_1:=\sum_{k=0}^{n-r} 2^{-ksp/q}\lambda_k^{1-p/q}, \quad \hbox{and}\quad 
\Sigma_2:=\sum_{k=0}^{n-r} \|\eta_{\lambda_k}\|^*_q.
\ee 

In order to give our choice for the thresholds $\lambda_k$, $0\le k\le n-r$, we  define $k^*$   as the integer which satisfies 
\be\label{defk}
2^{k^*-1}\leq 
\e^{-1/s}< 2^{k^*}. 
\ee
Note that $k^*>0$ since $\e<1$. 
We now define the thresholds 
\begin{equation}\label{choicelambdak}
    \lambda_k := 2^{-k^*s}\begin{cases}
        0, & k \leq k^*,\\
        2^{\beta(k-k^*)}, & k > k^*,
    \end{cases}
\end{equation}
where $\beta$ is any parameter from the interval $\beta\in \left(\frac{d}{2},\frac{sp}{q-p}\right)$.  This interval is non-empty due to assumption \eqref{parameters}. The (implied) constants in what follows may depend upon the choice of $\beta$, but $\beta$ will be specified later in the proof and will depend only upon $\alpha$. %

Note that it follows from \eref{variance-equation-11571} and \eref{defk} that  
 \be 
\label{sigmak*}
\epsilon\asymp \sigma_{k^*}, \quad \sigma_k\asymp 2^{\frac{d}{2}(k-k^*)}\varepsilon,\quad 0\leq k\leq n-r,
\ee
and
\be\label{sigma}
\lambda_k\asymp2^{\beta(k-k^*)}\varepsilon,\quad 
\hbox{when}\quad k>k^*. 
\ee

In the following subsections, we will bound the deterministic  term
$\Sigma_1$ and the stochastic term $\Sigma_2$ for the above choice of thresholds and thereby prove Theorem \ref{T:mainprob}.

\subsection{Bounding the Deterministic Term \texorpdfstring{$\Sigma_1$}{Sigma1}}
\label{SS:bdet1}

Using the above choice of $\lambda_k$, we have
\begin{eqnarray} 
\label{boundSigma1}
    \Sigma_1&=&\sum_{k=k^*+1}^{n-r} 2^{-ksp/q} \lambda_k^{\left(1-p/q\right)}
    \lsim 2^{-k^*s}\sum_{k > k^*} 2^{\left (\beta\left(1-\frac{p}{q}\right) - s\frac{p}{q}\right)(k-k^*)}%
    \lesssim  2^{-k^*s}\lesssim %
    \e.
\end{eqnarray}
Note that  the middle sum is a convergent geometric series because $\beta < \frac{sp}{q-p}$. In the last $\lesssim$ we have used the definition \eref{defk} of $k^*$.

\subsection{Bounding the Stochastic  Term \texorpdfstring{$\Sigma_2$}{Sigma2}}
\label{SS:stoch1}
We turn now to bounding the stochastic term $\Sigma_2$.  Our goal is to show that for $f\in K:=U(B_\infty^s(L_p(\Omega)))$, we have
\be 
\label{goal}
\mathbb{P}(\Sigma_2> t \e)\le Ce^{-ct^\alpha},\quad  t\ge 1,
\ee
for any $0<\alpha<2-\frac{d(q-p)}{sp}$ which will in turn prove the theorem.  Since we are in the primary case, we have $2-\frac{d(q-p)}{sp}>0$.
 The constants $c,C>0$ depend only on $s,p,q,d$, and the choice of $\alpha$.
In this section, we will frequently make use of the following lemma, whose proof is in Appendix~\ref{LemmaK2}.

\begin{lemma}
\label{L:known2}
  For any $0\le a \le a_0$; $b\ge b_0>0$; $ c> 0 $, and $\tau\ge 1$, we have
\be\label{LL}
 \sum_{k\ge 0}  2^{a k}e^{-c2^{b k}\tau} \le C(a_0,b_0,c)e^{-c\tau},\quad \tau\ge 1.
\ee
\end{lemma}

We fix any $q\ge1$ and $t\ge 1$.   Recall that according to Remark \ref{R:t} we only have to prove \eref{goal} for such $t$. We introduce numbers $t_k \geq 0$, to be specified below, which satisfy $\sum_{k=0}^{n-r}t_k\le t$. A standard union bound argument implies that it suffices to show that
\be
\sum_{k=0}^{n-r}\mathbb{P}(\|\eta_{\lambda_k}\|_q^*>t_k\e) \le Ce^{-ct^\alpha},\quad 
t\ge 1,
\ee
where we choose 
\begin{equation}
t_k:= \bar ct\begin{cases}
    2^{\delta (k -k^*)}, &0\leq k\leq k^*,\\
    2^{\delta(k^*-k)}, &k^*<k\leq n-r,
\end{cases}
\end{equation}
with  $0<\delta<d/2$,  
and define $\bar c$  so that  that $\sum_{k=0}^{n-r} t_k= t$.  Note that 
$\bar c\geq c(\delta)>0$, where  $c(\delta)$ is a constant depending  only $\delta$. The subsequent constants in what follows may depend upon $\delta$, but $\delta$ will be specified later in the proof and will depend only upon $\alpha$.

The remainder of the proof is divided into the following two cases.

\noindent 
{\bf Case 1:}  A bound for $\sum_{k=0}^{k^*}\mathbb{P}(\|\eta_{\lambda_k}\|_q^*>t_k\e)$.

Let $0\le k\le k^*$.  In this case, all thresholds are zero, i.e., no thresholding takes place.
We return to Theorem \ref{T:mainthresh1} and take $T=T_k:=t_k\e =\bar c2^{\delta(k-k^*)}t\e \geq c(\delta)2^{\delta(k-k^*)}t\e$.  Since the thresholds are zero, we have  $b=b_k= 2^{-1/q}T_k $ 
and $\sigma_k\asymp 2^{\frac{d}{2}(k-k^*)}\e$.   We compute
\be 
\label{cas13}
T_k^2/\sigma_k^2 \gtrsim  t^2\e^2   2^{2(k-k^*)\delta}2^{d(k^*-k)}  \e^{-2}\asymp t^2\mu^{k^*-k},\quad \mu:=2^{d-2\delta}>1.
\ee 
Thus, if 
{$t\geq 1$}, 
then 
$T_k>C\sigma_k$, $0\le k\le k^*$.
Theorem \ref {T:mainthresh1} now gives
\be
\label{case1}
   \mathbb{P}(\|\eta_{\lambda_k}\|_q^*>t_k\e)  \le  CT_k^{-q}
   { {\sigma_k}^q}e^{-
c( \sigma_k^{-1}T_k)^2}\le %
Ce^{-c\mu^{(k^*-k)}t^2},\quad 0\le k\le k^*.
   \ee 
    Hence, using Lemma \ref{L:known2} with $\tau=t^2\geq 1$, we obtain the bound
    \be 
\label{case14}
\sum_{k=0}^{k^*}\mathbb{P}(\|\eta_{\lambda_k}\|_q^*>t_k\e)\le C \sum_{k=0}^{k^*} e^{-c\mu^{(k^*-k)}t^2}\le    Ce^{-ct^2}.
\ee 
Here, we have changed $c,C$ from line to line, but ultimately, we will end up with fixed $c,C,$ once $\delta$ and $\beta$ are specified, depending on the choice of $\alpha$.

\noindent 
{\bf Case 2:} A bound for $\sum_{k>k^*}\mathbb{P}(\|\eta_{\lambda_k}\|_q^*>t_k\e)$.

We next consider the case when $k> k^*$.  In this case, the thresholds are given in \eref{choicelambdak} and thereby satisfy $\lambda_k\asymp \e 2^{\beta(k-k^*)}$, where we recall that $\beta\in \left(\frac{d}{2},\frac{sp}{q-p}\right)$.
We fix $k$ and proceed to bound each of the probabilities $\mathbb{P}(\|\eta_{\lambda_k}\|_q^*>t_k\e)$. 
We use Theorem  \ref{T:mainthresh1} with 
 $T=T_k:=t_k\e =\bar c2^{-\delta(k-k^*)}t\e \geq c(\delta)2^{-\delta(k-k^*)}t\e$, and thus
\begin{equation}
b=b_k=\max\{2^{-1}\lambda_k,2^{-1/q}T_k\}\gtrsim \e \max\{2^{\beta(k-k^*)}, t2^{-\delta(k-k^*)}\}.
\end{equation}
Since $\sigma_k\asymp 2^{\frac{d}{2}(k-k^*)}\e$, we have that 
\be 
\label{bk}
\sigma_k^{-1}b_k\gtrsim 2^{-(k-k^*)d/2} \max\{2^{\beta(k-k^*)}, t2^{-\delta(k-k^*)}\}.
\ee
Applying Theorem \ref{T:mainthresh1}, we obtain the bound
\begin{eqnarray}
   \mathbb{P}(\|\eta_{\lambda_k}\|_q^*>t_k\e)&\lsim& T_k^{-q}{\sigma_k^q}{\rm exp}(-c
( \sigma_k^{-1}b_k)^2/4) \nonumber\\
   &\lsim&  (t\e)^{-q}2^{{(k-k^*)}\delta q} \e^q 2^{(k-k^*)qd/2} {\rm exp}(-c
(\sigma_k^{-1}b_k)^2/4) \nonumber\\
   &\lsim&  \left[t^{-1}2^{(k-k^*)(d/2+\delta)}\right]^{q} {\rm exp}(-c
(\sigma_k^{-1}b_k)^2/4).
\label{probability-bound-1350}
\end{eqnarray}
The maximum on the right hand side in \eqref{bk} will be given by the second term for $k^*<k\le k'$ and  by the first term for $k> k'$, where the value of $k'=k'(t)$ satisfies
\be\label{k'}
t^{\frac{1}{\beta+\delta}} \lsim 2^{(k'-k^*)}\leq t^{\frac{1}{\beta+\delta}}.
\ee
Plugging this into \eqref{probability-bound-1350} gives
\begin{eqnarray}
    \label{semifinal-probability-bound}
\mathbb{P}(\|\eta_{\lambda_k}\|_q^*>t_k\e) &\lsim&
2^{(k-k^\prime)q(d/2+\delta)}
2^{(k^*-k')q(\beta-d/2)}{\rm exp}(-c
(\sigma_k^{-1}b_k)^2/4)
\nonumber \\
&\lsim&
2^{(k-k^\prime)q(d/2+\delta)}{\rm exp}(-c
(\sigma_k^{-1}b_k)^2/4),
\end{eqnarray}
where we have used that $\beta>d/2$.
Now we write
\begin{eqnarray} 
\label{writedec}
\quad \sum_{k>k^*}\mathbb{P}(\|\eta_{\lambda_k}\|_q^*>t_k\e)=\!\!\!\!\!\!\!\sum_{k^*<k\le k'}\mathbb{P}(\|\eta_{\lambda_k}\|_q^*>t_k\e)+\!\!\!\sum_{k>k'}\mathbb{P}(\|\eta_{\lambda_k}\|_q^*>t_k\e)%
=:\Sigma'+\Sigma'',
\end{eqnarray} 
and bound each of these sums.  

{\bf Case 2.1:} A bound for $\Sigma'$.

It follows from \eref{bk} that in the case 
$k^*<k\le k'$,
\be 
\label{case22}
\sigma_k^{-1}b_k
\gtrsim  2^{-(k-k^*)d/2}\max\{2^{\beta(k-k^*)}, t2^{-\delta(k-k^*)}\} =  2^{-(k-k^*)(d/2+\delta)}t.
\ee
Plugging this into \eqref{semifinal-probability-bound} gives
\be
\begin{split}
    \mathbb{P}(\|\eta_{\lambda_k}\|_q^*>t_k\e) &\lsim 2^{(k-k^\prime)q(d/2+\delta)}{\rm exp}(-c2^{-(k-k^*)(d+2\delta)}t^2)\\
    &\lsim {\rm exp}(-c2^{-(k-k^*)(d+2\delta)}t^2)
\end{split}
\ee
since $k \leq k'$. Therefore, we see that
\begin{eqnarray}
\label{fb1}
\Sigma'&=&\sum_{k^*<k\le k'}\mathbb{P}(\|\eta_{\lambda_k}\|_q^*>t_k\e)
\lsim \sum_{k=0}^{k'-k^*}
 {\rm exp}(-c2^{-k(d+2\delta)}t^2) \nonumber \\
 &\lsim& \sum_{k=0}^\infty {\rm exp}(-c2^{k(d+2\delta)}[2^{-(k'-k^*)(d + 2\delta)}t^2])
 \lsim e^{-c\tau} \leq e^{-ct^{\frac{2\beta - d}{\beta + \delta}}},
\end{eqnarray}
where we have used  Lemma \ref{L:known2} with $\tau = 2^{-(k'-k^*)(d + 2\delta)}t^2$. Note that by \eqref{k'} it follows that 
$\tau \geq t^{\frac{2\beta - d}{\beta + \delta}} \geq 1$.

{\bf Case 2.2:} A bound for $\Sigma''$. 

 We obtain from  \eref{bk} in the case 
$k> k'$ that
 \be
\sigma_k^{-1}b_k
\gtrsim  2^{-(k-k^*)d/2}\max\{2^{\beta(k-k^*)}, t2^{-\delta(k-k^*)}\} =  2^{(k-k^*)(\beta-d/2)}.
 \ee
 Using  \eqref{semifinal-probability-bound}, we derive 
 \be
    \mathbb{P}(\|\eta_{\lambda_k}\|_q^*>t_k\e) \lsim 2^{(k-k^\prime)q(d/2+\delta)}{\rm exp}(-c2^{(k-k^*)(2\beta-d)}),
\ee
which gives the estimate
   \begin{eqnarray} 
\nonumber
\Sigma''&=&\sum_{k>k'}\mathbb{P}(\|\eta_{\lambda_k}\|_q^*>t_k\e)\lsim \sum_{k>k'}
 2^{(k-k^\prime)q( d/2{+\delta})} {\rm exp}(-c2^{(k-k^*)(2\beta-d) })
\nonumber \\
& = &  \sum_{j>0}
 2^{jq(d/2{+\delta})} {\rm exp}(-c2^{j(2\beta-d)}[2^{(k^\prime - k^*)(2\beta-d)}])\nonumber\\
&\lsim&  {\rm exp}(-c2^{(k'-k^*)(2\beta-d) })\lsim {\rm exp}( - ct^{\frac{2\beta  -d}{\beta+\delta}}).
\end{eqnarray} 
Here, we again apply Lemma \ref{L:known2} with 
   $\tau=2^{(k^\prime - k^*)(2\beta-d)}\geq 1$.

Combining the bounds 
 for $\Sigma'$ and $\Sigma''$, we obtain that 
\be
\label{sigma2}
\sum_{k>k^*}\mathbb{P}(\|\eta_{\lambda_k}\|_q^*>t_k\e)\lsim {\rm exp}( - ct^{\frac{2\beta  -d}{\beta+\delta}}), \quad t\geq 1.
\ee
Recall that $\beta$ and $\delta$ can be chosen as any numbers satisfying 
$\beta\in (\frac{d}{2},\frac{sp}{q-p})$ and  
  $\delta\in (0,d/2)$, respectively.    Therefore, given any 
 $\alpha$ satisfying
 $0<\alpha<2-\frac{d(q-p)}{sp}$, we   can choose $\beta$ and 
$\delta$ so that 
\begin{equation}
 \alpha= \frac{2\beta  -d}{\beta+\delta}%
 \quad 
 \hbox{since}
 \quad 
 \lim_{\substack{\beta\to sp/(q-p)\\\delta\to 0}}\,\, \frac{2\beta  -d}{\beta+\delta}=2-\frac{d(q-p)}{sp}.
\end{equation}
 This, together with  \eref{case14}, proves  \eref{goal},  and completes the proof.

\section{Proof of Theorem \ref{T:mainexplb} }
\label{S:lowerbounds}

The proof we present here uses well-known ideas from both optimal recovery \cite{BDPS,NT,V, KNS} and minimax estimation \cite{barron1999risk,yang1999information,DJ}. We give a complete and self-contained argument which unifies both regimes. We restrict our presentation to the case where the parameters $p,q,s$ are in the primary case (see \eref{parameters}). We fix these parameters for the remainder of this section. We also fix $m\ge 1$ and the variance $\sigma^2$ assumed on the noise vectors. 
We let the data sites $\cX=\{x_i\in\overline \Omega, \ i=1,\dots,m\}$ be arbitrary but fixed. 

Let $K=U(B_\tau^s(L_p(\Omega)))$, where $\Omega=(0,1)^d$, $0<\tau\le \infty$, and   the other parameters are in the primary region.  
Our goal in this section is to prove the lower bound
\be 
\label{goal2}
R_m(K,\sigma,\cX)_q\ge c\left\{m^{-\frac{s}{d}+(1/p-1/q)_+}+
\min\left\{1, \left[\frac{\sigma^2}{m}\right]^{\frac{s}{2s+d}}\right\}\right\}.
\ee 
  We will use without further mention the fact that
the left side of \eref{goal2} is increasing as $\sigma$ increases and decreasing as $m$ increases. This fact shows in particular that it is enough to prove \eref{goal2} for $\sigma\le\sqrt{m}$ which is an assumption we impose for the remainder of this section. 
Therefore,
\be 
\label{epsagain}
\min\left\{1, \left[\frac{\sigma^2}{m}\right]^{\frac{s}{2s+d}}\right\}= \left[\frac{\sigma^2}{m}\right]^{\frac{s}{2s+d}}=:\e.
\ee 
Throughout this section constants like $c,C$ depend 
only on $p,q,s$ and $\tau$,  but not on $\sigma$ or $m$, and may change from line to line.

In going further in this section, we denote by $y$ any vector in $\R^m$, by $y(f)$ the true data vector of $f$ and by $\tilde y(f)$
 the noisy observation vector 
 \be
 \label{noisyy}
 y(f):=(f(x_1),\dots,f(x_m)),\quad \tilde y(f):=y(f)+\eta,
 \ee 
 where the entries in the noise vector $\eta$ are i.i.d. $\cN(0,\sigma^2)$, and hence $\tilde y(f)$ is a random variable.
 
Let $A$ be any mapping from $\R^m$ into $L_q(\Omega)$.  For any such mapping $A$, we will show that there is a function $f:=f_A \in K$ for which
\be 
\label{goal3}
\mathbb E\|f-A(\tilde y(f))\|_{L_q(\Omega)}\ge c\{m^{-\frac{s}{d}+(1/p-1/q)_+}+\e\},
\ee 
which will then complete the proof of the lower bound.

As already noted  in this paper, see \cite{BDPS},
it is known from optimal recovery that there exist two  functions $f,g\in K$ such that $y(f)=y(g)$ and
\be 
\label{ORlb}
cm^{-\frac{s}{d}+(1/p-1/q)_+}\le \|f-g\|_{L_q(\Omega)}.
\ee
Since $A(\tilde y(f))=A(\tilde y(g))$, it follows that
\be 
\label{ORlb1}
cm^{-\frac{s}{d}+(1/p-1/q)_+}\le   \|f-A(\tilde y(f)\|_{L_q(\Omega)}+\|g-A(\tilde y(g)\|_{L_q(\Omega)} .
\ee
If we now take an expectation with respect to the noise draw, we have
\be 
\label{ORlb2}
\max\{ \mathbb E\|f-A(\tilde y(f))\|_{L_q(\Omega)}, \mathbb E\|g-A(\tilde y(g))\|_{L_q(\Omega)}\}\ge cm^{-\frac{s}{d}+(1/p-1/q)_+}.
\ee
We have satisfied \eref{goal3} whenever $\e\le m^{-\frac{s}{d}+(1/p-1/q)_+} $.  Therefore, in going further, we can   assume that 
$\sigma,m$ satisfy
\be 
\label{assume1}
\left [\frac{\sigma^2}{m}\right ]^{\frac{s}{2s+d}}=\e \ge m^{-\frac{s}{d}+(1/p-1/q)_+}.
\ee 
We fix such a pair $\sigma,m$ for the remainder of this section and proceed to show the existence of a function $f= f_A\in K$ such that
\be 
\label{goal4} \mathbb E\|f-A(\tilde y(f))\|_{L_q(\Omega)} \ge c \e, 
\ee
which will complete the proof of the lower bound.

To produce such a function $f\in K$ satisfying \eref{goal4}, we proceed in a  way, similar  to the proof of lower bounds for the entropy of $K$.  We want to produce a family $\cF=\{f_1,\dots, f_N\}\subset K$  of functions with $N$ large  such that $\|f_i-f_j\|_{L_q(\Omega)}\ge \e$.
We recall that the entropy numbers $\e_\ell(K)_{L_q(\Omega)}$,  $\ell\ge 1$, of $K$  are known to satisfy $\e_\ell(K)_{L_q(\Omega)}\asymp \ell^{-\frac{s}{d}}$, $\ell\ge 1$, see \cite{BDPS}.
This means that for each integer $\ell\ge 1$,  there is a set of $2^\ell$ functions $g_1,\dots,g_{2^\ell}$ all from $K$ such that 
\be 
\label{packing}
\|g_i-g_j\|_{L_q(\Omega)}\ge c\ell^{-\frac{s}{d} }, \quad i\neq j=1, \ldots,2^\ell.
\ee

We now define  $n$ as the smallest integer such that
\be 
\label{defn}
n^{-s}\le \e,
\ee
and describe the construction of $\cF:=\{f_1,\dots,f_N\}$. We refer the reader to \S 12 of \cite{BDPS}, where a similar construction is given in detail.
Let $\phi\in C^\infty_c([0,1]^d)$ be a fixed smooth and compactly supported function on the unit cube $[0,1]^d$ with $\| \phi\|_{L_\infty(\Omega)}= 1$ and $M:=\|\phi\|_{B_\tau^s(L_p(\Omega))}$ finite. It follows that $\phi\ge c$ on a set of positive measure, and therefore
\be 
\label{Lqphi}
\|\phi\|_{L_q(\Omega)}\ge c.
\ee
Consider a uniform (tensor product) subdivision of $\Omega$ into $n^d$ cubes $Q_1,\dots Q_{P}$, $P:=n^d$  of side length $1/n$, where $n$ is given by \eref{defn}. 
We define the functions
\be
    \phi_i(x) =  \gamma n^{-s}\phi(n(x-z_i)), \quad i=1,\dots,P,
\ee
where $z_i$ is the bottom left corner of the cube $Q_i$ and the normalizing constant $\gamma$ defined momentarily. Thus, each $\phi_i$ is a rescaling of $\phi$ to the cube $Q_i$ and each $\phi_i$ has $L_\infty(\Omega)$ norm equal to $\gamma n^{-s}$.

To construct our  collection of functions  $\cF$, we need a well known combinatorial lemma (see for instance Lemma 2.2 in Chapter 15 in \cite{lorentz1996constructive}) stating that  there is a subset $S\subset \{\pm 1\}^P$ such that:
\begin{enumerate}
    \item $|S| \geq 2^{cP}$
    \item $\|\kappa - \nu\|_{\ell^1} \geq cP$ for any $\kappa\neq \nu$ with $\kappa,\nu\in S$.
\end{enumerate}
We then consider
\be
    \mathcal{F}= \left\{f=\sum_{i=1}^P\ \kappa_i\phi_i \: : \: \kappa=(\kappa_1,\dots,\kappa_P) \in S\right\},
\ee
and define 
\be 
\label{defN}
N:=\#(\cF)\ge 2^{cP}=2^{cn^d}.
\ee 
Note that if the normalizing constant $\gamma$ is small enough (depending only on $p,q,s$ and $\tau$), we can guarantee that  $\cF\subset K$.   This now fixes $\gamma$.   

There are $N$ functions $f_1,\dots,f_N$ in $\cF$ and each has $L_\infty(\Omega)$ norm at most $c_0\gamma \e$ for a fixed constant $c_0$.  Additionally, if $i\neq j$ then
the vectors $\kappa$ giving the sign pattern for $f_i$ and $\nu$ giving the sign pattern for $f_j$ will  differ in at least $cP$ coordinates, that is, $|\kappa_\ell-v_\ell|=2$ for at least $cP$ coordinates.   This   gives  (see \eref{Lqphi}) that
\be 
\label{separate}
\|f_i-f_j\|_{L_q(\Omega) }\ge c_0\gamma\e, \quad i\neq j.
\ee
We also need a bound on the values of the $f_i$ at the $m$ data sites.  Since $\|f_i\|_{L_\infty(\Omega)}\le \gamma n^{-s}$ and from \eref{defn}, we have that
$y(f_i)=(f(x_1),\dots,f(x_m))$ satisfies
\be 
\label{boundyi}
\|y(f_i)\|\le C_0\gamma\sqrt{m}\e,\quad i=1,\dots,N,
\ee 
where $\|\cdot\|$ is the Euclidean norm on $\R^m$.  Notice that 
\be 
\label{fy1}
 \sigma^{-2}m\e^2 = \left[\frac{\sigma^2}{m}\right]^{\frac{2s}{2s+d}-1}= \left[\frac{\sigma^2}{m}\right]^{\frac{-d}{2s+d}}=\e^{-d/s}\asymp n^d\quad \Rightarrow \quad \sqrt{m}\e\asymp \sigma n^{d/2},
\ee 
and therefore  the  bound \eref{boundyi} can be equivalently written as
\be 
\|y(f_i)\|\le C_1 \gamma\sigma n^{d/2}=C_1 \sigma \gamma\sqrt{\ln N},\quad i=1,\dots, N,
\ee
where the constant $C_1$ depends only on $p,q,s$ and $\tau$, which are of course fixed. 

Combining all of this, by choosing $\gamma$ appropriately small (but fixed depending only upon $p,q,s$ and $\tau$), we can guarantee that $\mathcal{F}\subset K$, $\|f_i-f_j\|_{L_q(\Omega) }\ge \widetilde{c}_0\e$ for $i\neq j$, and
\be
    \|y(f_i)\| < \sigma \sqrt{\ln (N/5)},\quad i=1,\dots, N.
\ee

We proceed now to show that given any algorithm $A:\R^m\to L_q(\Omega)$, there is a function $f\in K$ that satisfies \eref{goal4}.  For this purpose, we consider the sets
\be
    B_i = A^{-1}\left(\left\{g\in L_q(\Omega):~\|f_i - g\|_{L_q(\Omega)} < \frac{\widetilde{c}_0}{2}\epsilon\right\}\right)\subset \R^m, \quad i=1, \ldots,N.
\ee
Obviously, the sets $B_i$ are disjoint due to the fact that the $f_i$ are separated by at least $\widetilde c_0\epsilon$. For $y\in \mathbb{R}^m$, $\sigma > 0$, and $B\subset \mathbb{R}^m$, we introduce the notation
\be
    \mu_{y,\sigma}(B) :=
        \frac{1}{(2\pi \sigma^2)^{m/2}}\int_B e^{-\|z-y\|^2/2\sigma^2} dz
\ee
for the measure of the set $B$ under a Gaussian centered at $y$ with variance $\sigma^2$.

Since the sets $B_i$ are disjoint, it follows that for some $i$ we must have 
\be 
\label{muusthave} 
\frac{1}{(2\pi \sigma^2)^{m/2}}\int_{B_i} e^{-\|z\|^2/2\sigma^2} dz=\mu_{0,\sigma}(B_i) \leq \frac{1}{N}.
\ee 
We fix such a value of $i$. We can now show that the function $f_i$ satisfies \eref{goal4},
 The proof of this fact is typically completed in the statistics literature via a proper comparison between a Gaussian centered at $0$ and a Gaussian centered  at $y(f_i)$.   Such a comparison is provided   by using the KL-divergence.   For the reader's convenience, in case they are not familiar with these concepts, we state and explicitly prove the following lemma.

\begin{lemma}
\label{L:perturb}
     Suppose that $\sigma > 0$, $0 < \bar \alpha < 1/5$ and $y\in \mathbb{R}^m$ with $\|y\|^2 < -\sigma^2\ln(5\bar \alpha)$. If $B\subset \mathbb{R}^m$ is a measurable set satisfying
    \be
        \mu_{0,\sigma}(B) = \bar\alpha,
    \ee
    then
    \begin{equation}
        \mu_{y,\sigma}(B) < \frac{1}{2}.
    \end{equation}
\end{lemma}

We give a   measure-theoretic proof of this lemma in Appendix~\ref{SS:A2}. Using this lemma, we complete the proof of \eref{goal4} as follows. We take for $B$ the set $B_i$ which satisfies $\mu_{0,\sigma}(B_i) \leq 1/N$.  We know that 
 the vector $y(f_i)$ satisfies 
 \begin{equation}
 \|y(f_i)\|^2 < \sigma^2 \ln (N/5).
 \end{equation}
 Applying the lemma  with $\bar \alpha=1/N$, we find that
 \be 
 \label{find1}
 \mu_{y(f_i),\sigma}(B_i) < 1/2.
 \ee 
 This means that
 \be 
 \label{find2}
 \mu_{y(f_i),\sigma}(B^c_i) > 1/2.
 \ee 
For any $y\in B_i^c$, we have $\|f_i-A(y)\|_{L_q(\Omega)}\ge \frac{\widetilde c_0}{2}\e $. It follows that
\be 
 \label{find3}
\mathbb E \|f_i-A(\tilde y(f_i))\|_{L_q(\Omega)} \ge \mu_{y(f_i),\sigma}(B^c_i)\cdot  \frac{\widetilde c_0}{2}\e > \frac{\widetilde{c}_0}{4}\e.
 \ee 
This completes the proof of \eref{goal4} and thereby proves Theorem \ref{T:mainexplb}.
\hfill $\Box$

\section{Proof of  Theorem \ref{T:mainexp}}
\label{SS:pTexp}
The lower bound in Theorem \ref{T:mainexp} follows from Theorem \ref{T:mainexplb}. To prove the upper bound, we fix the parameters $s,p,q$ from the primary region.
Let
\begin{equation}
M:=Cm^{-\frac{s}{d} +  ( \frac{1}{p} - \frac{1}{q})_+}~~
\text{and}~~ M^*:=C\left[\frac{\sigma^2}{m}\right]^{\frac{s}{2s + d}},
\end{equation}
with $C$ the constant in Theorem \ref{T:mainprob}. We then have for every $f\in K$,
\begin{eqnarray} 
\label{pTexp}
\mathbb E\|f-\hat f\|_{L_q(\Omega)} &=&\int_0^\infty \mathbb{P}(\|f-\hat f\|_{L_q(\Omega)}\ge t)\,dt
\le M+\int_M^\infty \mathbb{P}(\|f-\hat f\|_{L_q(\Omega)}\ge t)\,dt\nonumber\\
&= & M+M^*\int_0^\infty  \mathbb{P}(\|f-\hat f\|_{L_q(\Omega)}\ge M+M^*t)\, \,dt \nonumber\\
&\le& M+CM^*\int_0^\infty  e^{-ct^\alpha}\,dt\le 
C\left(m^{-\frac{s}{d} +  ( \frac{1}{p} - \frac{1}{q})_+}+\left[\frac{\sigma^2}{m}\right]^{\frac{s}{2s + d}}\right),
\end{eqnarray}
where we have used Theorem \ref{T:mainprob}. This completes the proof in the case  when $\sigma^2\leq m$.  If $\sigma^2>m$, one can just take the algorithm to be $Ay=0$ for every $y\in \R^m$. In this case, the error is always bounded by a constant since $K$ embeds into $L_\infty(\Omega)$.
\hfill $\Box$

\section{Concluding Remarks}
\label{S:concludingremarks}
 The goal of this paper was to merge the optimal recovery and minimax theories for error measured in $L_q(\Omega)$, $\Omega=(0,1)^d$, $1\leq q<\infty$, for model classes $K$ that are the unit ball of Besov spaces.   We accomplished this goal by proving upper and lower bounds in Theorem \ref{T:mainexp} for $R_m(K,\sigma)_q $ in which the constants do not depend on either $\sigma$
 or $m$. As part of meeting this goal, we have given performance results in probability which are stronger than those in expectation. Our algorithm for accomplishing these results
 is also interesting in that it uses only piecewise polynomial approximation and not wavelets. This should be useful in extending our results to more general domains $\Omega$.
 
The above analysis of upper bounds on performance of learning algorithms in probability or expectations was carried out for Besov classes when the parameters
are in the primary region (see \eref{parameters}). We restricted our analysis to this case in order to keep the presentation as simple as possible.  A similar analysis can be carried out when the parameters are outside the primary region. We do not give the details in this case but mention the form of the results.

\begin{remark}
    Let $K=U(B_\tau^s(L_p(\Omega)))$ with $s>d/p$ and  $1\leq q\le\infty$. Outside of the primary regime (see \eqref{parameters}), one can use the techniques of this paper to prove that if $\sigma^2 \leq m/2$, then
    \be
        R_m(K,\sigma)_q \asymp \left\{m^{-\frac{s}{d} + \left(\frac{1}{p} -\frac{1}{q}\right)_+}+\left(\log(m/\sigma^2)\left [\frac{\sigma^2}{m}\right]\right)^{\frac{(s/d+1/q-1/p)}{(1+2(s/d-1/p))}}\right\},
    \ee
    see, for comparison, \cite{DeJu,AN}.
\end{remark}

\section*{Acknowledgments}
The authors thank Peter Binev, Albert Cohen, Wolfgang Dahmen, and Ryan Tibshirani for valuable conversations about the material in this paper.

\section*{Funding}
This research was supported in part by the NSF grants DMS-2134077 and DMS-2134140 of the NSF MoDL program (RD, RN, and GP), NSF grant DMS-2424305 (JS), and the ONR MURI grant
N00014-20-1-2787 (RD, RN, GP, and JS).

\bibliographystyle{plain}
\bibliography{refs}

\begin{thebibliography}{10}

\bibitem{barron1999risk}
Andrew Barron, Lucien Birg{\'e}, and Pascal Massart.
\newblock Risk bounds for model selection via penalization.
\newblock {\em Probability Theory and Related Fields}, 113:301--413, 1999.

\bibitem{binev2007universal}
Peter Binev, Albert Cohen, Wolfgang Dahmen, and Ronald DeVore.
\newblock Universal algorithms for learning theory. {Part II:} piecewise
  polynomial functions.
\newblock {\em Constructive Approximation}, 26:127--152, 2007.

\bibitem{binev2005universal}
Peter Binev, Albert Cohen, Wolfgang Dahmen, Ronald DeVore, Vladimir Temlyakov,
  and Peter Bartlett.
\newblock Universal algorithms for learning theory. {Part I:} piecewise
  constant functions.
\newblock {\em Journal of Machine Learning Research}, 6(9), 2005.

\bibitem{BDPS}
Andrea Bonito, Ronald DeVore, Guergana Petrova, and Jonathan~W. Siegel.
\newblock Convergence and error control of consistent {PINNs} for elliptic
  {PDEs}.
\newblock {\em arXiv preprint arXiv:2406.09217}, 2024.

\bibitem{chambolle1998nonlinear}
Antonin Chambolle, Ronald~A. DeVore, Nam-Yong Lee, and Bradley~J. Lucier.
\newblock Nonlinear wavelet image processing: Variational problems,
  compression, and noise removal through wavelet shrinkage.
\newblock {\em IEEE Transactions on Image Processing}, 7(3):319--335, 1998.

\bibitem{chaudhuri1994piecewise}
Probal Chaudhuri, Min-Ching Huang, Wei-Yin Loh, and Ruji Yao.
\newblock Piecewise-polynomial regression trees.
\newblock {\em Statistica Sinica}, pages 143--167, 1994.

\bibitem{CDD}
Albert Cohen, Wolfgang Dahmen, and Ronald DeVore.
\newblock Multiscale decompositions on bounded domains.
\newblock {\em Transactions of the American Mathematical Society},
  352(8):3651--3685, 2000.

\bibitem{CDV}
Albert Cohen, Ingrid Daubechies, and Pierre Vial.
\newblock Wavelets and fast wavelet transforms on the interval.
\newblock {\em Applied and Computational Harmonic Analysis}, 1(1):54--81, 1993.

\bibitem{DeJu}
Bernard Delyon and Anatoli Juditsky.
\newblock On minimax wavelet estimators.
\newblock {\em Applied and Computational Harmonic Analysis}, 3(3):215--228,
  1996.

\bibitem{DL}
Ronald DeVore and George G.~Lorentz.
\newblock {\em Constructive Approximation}.
\newblock Grundlehren, Springer Verlag (Berlin), Vol. 303. Springer, Berlin,
  1993.

\bibitem{DSmono}
Ronald DeVore and Robert Sharpley.
\newblock Maximal functions measuring smoothness.
\newblock {\em Memoirs of the AMS}, 293, 1984.

\bibitem{DP}
Ronald~A DeVore and Vasil~A Popov.
\newblock Interpolation of besov spaces.
\newblock {\em Transactions of the American Mathematical Society},
  305(1):397--414, 1988.

\bibitem{DS}
Ronald~A. DeVore and Robert~C. Sharpley.
\newblock Besov spaces on domains in {$\mathbb{R}^d$}.
\newblock {\em Transactions of the American Mathematical Society},
  335(2):843--864, 1993.

\bibitem{donoho1994ideal}
David~L. Donoho and Iain~M. Johnstone.
\newblock Ideal spatial adaptation by wavelet shrinkage.
\newblock {\em Biometrika}, 81(3):425--455, 1994.

\bibitem{donoho1994minimax}
David~L. Donoho and Iain~M. Johnstone.
\newblock Minimax risk over {$l_p$}-balls for {$l_q$}-error.
\newblock {\em Probability Theory and Related Fields}, 99:277--303, 1994.

\bibitem{DJ}
David~L. Donoho and Iain~M. Johnstone.
\newblock Minimax estimation via wavelet shrinkage.
\newblock {\em Annals of Statistics}, 26(3):879--921, 1998.

\bibitem{DJKP}
David~L. Donoho, Iain~M. Johnstone, G{\'e}rard Kerkyacharian, and Dominique
  Picard.
\newblock Wavelet shrinkage: asymptopia?
\newblock {\em Journal of the Royal Statistical Society: Series B
  (Methodological)}, 57(2):301--337, 1995.

\bibitem{donoho1996density}
David~L. Donoho, Iain~M. Johnstone, G{\'e}rard Kerkyacharian, and Dominique
  Picard.
\newblock Density estimation by wavelet thresholding.
\newblock {\em Annals of Statistics}, pages 508--539, 1996.

\bibitem{donoho1990minimax}
David~L Donoho, Richard~C Liu, and Brenda MacGibbon.
\newblock Minimax risk over hyperrectangles, and implications.
\newblock {\em Annals of Statistics}, pages 1416--1437, 1990.

\bibitem{ghosh2025signal}
Shubhangi Ghosh, Yilin Guo, Haolei Weng, and Arian Maleki.
\newblock Signal-to-noise ratio aware minimax analysis of sparse linear
  regression.
\newblock {\em arXiv preprint arXiv:2501.13323}, 2025.

\bibitem{gine2016mathematical}
Evarist Gin{\'e} and Richard Nickl.
\newblock {\em Mathematical Foundations of Infinite-Dimensional Statistical
  Models}.
\newblock Cambridge Series in Statistical and Probabilistic Mathematics.
  Cambridge University Press, 2016.

\bibitem{guo2024note}
Yilin Guo, Shubhangi Ghosh, Haolei Weng, and Arian Maleki.
\newblock A note on the minimax risk of sparse linear regression.
\newblock {\em arXiv preprint arXiv:2405.05344}, 2024.

\bibitem{10374243}
Yilin Guo, Haolei Weng, and Arian Maleki.
\newblock Signal-to-noise ratio aware minimaxity and higher-order asymptotics.
\newblock {\em IEEE Transactions on Information Theory}, 70(5):3538--3566,
  2024.

\bibitem{hardle2012wavelets}
Wolfgang H{\"a}rdle, Gerard Kerkyacharian, Dominique Picard, and Alexander
  Tsybakov.
\newblock {\em Wavelets, approximation, and statistical applications}, volume
  129.
\newblock Springer Science \& Business Media, 2012.

\bibitem{johnstone1999wavelets}
Iain~M. Johnstone.
\newblock Wavelets and the theory of non-parametric function estimation.
\newblock {\em Philosophical Transactions of the Royal Society of London.
  Series A: Mathematical, Physical and Engineering Sciences},
  357(1760):2475--2493, 1999.

\bibitem{maxiset}
G\'erard Kerkyacharian and Dominique Picard.
\newblock Thresholding algorithms, maxisets and well-concentrated bases.
\newblock {\em Test}, 9(2):283--344, 2000.
\newblock With comments, and a rejoinder by the authors.

\bibitem{MinimaxMaxiset}
G{\'e}rard Kerkyacharian and Dominique Picard.
\newblock {Minimax or maxisets?}
\newblock {\em Bernoulli}, 8(2):219 -- 253, 2002.

\bibitem{kolaczyk2005multiscale}
Eric~D. Kolaczyk and Robert~D. Nowak.
\newblock Multiscale generalised linear models for nonparametric function
  estimation.
\newblock {\em Biometrika}, 92(1):119--133, 2005.

\bibitem{KNS}
David Krieg, Erich Novak, and Mathias Sonnleitner.
\newblock Recovery of {S}obolev functions restricted to iid sampling.
\newblock {\em Mathematics of Computation}, 91(338):2715--2738, 2022.

\bibitem{lorentz1996constructive}
G.G. Lorentz, M.~Golitschek, and Y.~Makovoz.
\newblock {\em Constructive Approximation: Advanced Problems}.
\newblock Grundlehren der mathematischen Wissenschaften. Springer Berlin
  Heidelberg, 1996.

\bibitem{ma2024high}
Tianyi Ma, Kabir~A Verchand, and Richard~J Samworth.
\newblock High-probability minimax lower bounds.
\newblock {\em arXiv preprint arXiv:2406.13447}, 2024.

\bibitem{morkisz2020complexity}
Pawe{\l}~M Morkisz and Leszek Plaskota.
\newblock Complexity of approximating {H}{\"o}lder classes from information
  with varying {G}aussian noise.
\newblock {\em Journal of Complexity}, 60:101497, 2020.

\bibitem{AN}
Arkadi Nemirovski.
\newblock Topics in non-parametric statistics.
\newblock {\em Ecole d’Et{\'e} de Probabilit{\'e}s de Saint-Flour}, 2000.

\bibitem{NT}
Erich Novak and Hans Triebel.
\newblock Function spaces in {L}ipschitz domains and optimal rates of
  convergence for sampling.
\newblock {\em Constructive approximation}, 23:325--350, 2006.

\bibitem{10.1214/24-AOS2446}
Reese Pathak, Martin~J. Wainwright, and Lin Xiao.
\newblock {Noisy recovery from random linear observations: Sharp minimax rates
  under elliptical constraints}.
\newblock {\em Annals of Statistics}, 52(6):2816 -- 2850, 2024.

\bibitem{storozhenko1978jackson}
E.~A. Storozhenko and P~Oswald.
\newblock Jackson's theorem in the spaces $l_p (r^k),~0< p< 1$.
\newblock {\em Siberian Mathematical Journal}, 19(4):630--656, 1978.

\bibitem{V}
Jan Vybıral.
\newblock Sampling numbers and function spaces.
\newblock {\em Journal of Complexity}, 23:773--792, 2007.

\bibitem{yang1999information}
Yuhong Yang and Andrew Barron.
\newblock Information-theoretic determination of minimax rates of convergence.
\newblock {\em Annals of Statistics}, pages 1564--1599, 1999.

\end{thebibliography}

\appendix

\section{Proof of Lemma \ref{L:polynorms}}
\label{Lemma22}
   
\begin{proof}
All of the (quasi-)norms appearing in the statement of the lemma are equivalent  because they are \mbox{(quasi-)norms} on $\cP_r$. Indeed, if $\|\cdot\|$ is any of these (quasi-)norms,  we have $\|Q\|=0$ for $Q\in\cP_r$ if and only if $Q$ is the zero polynomial.    Thus, the only issue to be addressed is to prove that the constants appearing in these equivalences can be chosen to depend only on the dimension of the space $\cP_r$ and $q_0$ and not on $q,I$ or $N$,  once $N^d>\rho$.  The fact that the constants do not depend on $I$ is a simple matter of rescaling  which we do not discuss further.  So in proving the lemma,  we can assume that
$I=\Omega=(0,1)^d$.

It is well known, see for instance (3.1) in \cite{DS}, that the (quasi-)norms in (i) (which do not depend on $N$) are all uniformly equivalent for different values of $q \geq q_0$.
Similarly, the (quasi-)norms in (iii) are all uniformly equivalent for different values of $q \geq q_0$ (since these (quasi-)norms are simply on $\R^\rho$ and $\rho$ is fixed), and hence equivalent to the $q=2$ norm. From the orthogonality, it follows that the norms in  (ii) and (iii) are equal when $q = 2$. 

We will complete the proof by showing that (i) and (ii) are equivalent (uniformly in $q \geq q_0$ and $N^d > \rho$).   We denote  the cubes in the grid $\Lambda_I$ by $J$ and their collection by $\cJ_N$.
     For each $J\in\cJ_N$, we denote by  $x_J$    the lower left corner of $J$.  If  $Q\in\cP_r$, we define
$$
S_N(Q):=\sum_{J\in\cJ_N}Q(x_J)\chi_J,
$$ 
where $\chi_J$ is the characteristic function of $J$.
Then, $S$ is a piecewise constant function. From Markov's inequality  for polynomials, if $Q\in\cP_r$, we have
$$
|Q(x)-Q(x_J)|\le C(r) N^{-1}\|Q\|_{L_\infty(I)} \le C(r)N^{-1}\|Q\|_{L_q(I)},\quad x\in J.
$$ 
Here we note that the constant $C(r)$ can be chosen uniformly in $q_0 \leq q \leq \infty$. Therefore, we have
$$ 
\|Q-Q(x_J)\|_{L_q(J)}\le C(r) \|Q\|_{L_q(I)} N^{-1}|J|^{1/q}, \quad J\in\cJ_N,
$$
and hence
$$
\|Q-S_N(Q)\|_{L_q(I)} \le C(r)\|Q\|_{L_q(I)}N^{-1}.
$$
It follows that
$$ 
|\|Q\|_{L_q(I)} -\|S_N(Q)\|_{L_q(I)}|\le C(r)\|Q\|_{L_q(I)}N^{-1}.
$$
We now choose $N_0$ so that $C(r)N_0^{-1}\le 1/2$. 
Then, for $N\ge N_0$, we obtain the equivalency of $\|S_N(Q)\|_{L_q(I)}$ with $\|Q\|_{L_q(I)}$ with fixed equivalency constants.  This gives the equivalence of $\|(Q(x_j))\|^*_{\ell_q}$ and $ \|Q\|^*_{L_q(I)}$ since 
$$
\|S_N(Q)\|_{L_q(I)}=\|(Q(x_j))\|^*_{\ell_q} {\rm ~and~} \|Q\|_{L_q(I)}=\|Q\|^*_{L_q(I)}.
$$
On the other hand, we have the equivalence of the (quasi-)norms with $N\le N_0$ with fixed equivalency constants because there are only a finite number of them. The uniformity in $q_0 \leq q \leq \infty$ is obtained because the constants of equivalence are continuous functions of $1/q$ for fixed $N$, and hence these constants achieve a finite maximum on the compact interval $[0,1/q_0]$. In summary, we obtain that all the (quasi-)norms in (i) and (ii) are equivalent on $\cP_r$ with equivalency constants only depending on $r$.

Finally, let us use the norm in  (i) with $q=\infty$  and the norm in (iii) with  $q=2$, for  $Q=Q_{I,j}$, $j=1, \ldots,\rho$ (recall the $Q_{I,j}$'s  form an orthonormal basis for $\cP_r$ in the Hilbert space $L_2(\mu_I)$). We have shown that these two norms are equivalent, from which we obtain that  
 $$
 1\asymp \|Q_{I,j}\|_{L_\infty(I)},
 $$
 which proves \eref{QIjbound}.
 \end{proof}

\section {Proof of Lemma \ref{L:known}}
\label{LemmaK}

\begin{proof}
Note that if $a < 1$, there is nothing to prove. For $a \geq 1$, the following holds
\begin{eqnarray}
\nonumber
\int_a^\infty t^qe^{-t^2/2}\,dt&=&
\int_a^\infty (t^{q-1}e^{-t^2/4})(te^{-t^2/4})\,dt\nonumber\\
&\leq&
\left(\max_{t \geq 1} t^{q-1}e^{-t^2/4}\right)\int_a^\infty te^{-t^2/4}\,dt
\leq
C_qe^{-a^2/4}.
 \nonumber \quad\quad\qedhere
\end{eqnarray}
\end{proof}

\section{Proof of Lemma \ref{L:known2}}
\label{LemmaK2}
\begin{proof}
First, let us observe that  
\begin{equation}
 \sum_{k\ge 0}  2^{a k}e^{-c2^{b k}\tau} \le \sum_{k\ge 0}  2^{a_0 k}e^{-c2^{b_0 k}\tau}, \quad \tau\ge 1.
\end{equation}
Note that  there is a $
\bar k=\bar k(a_0,b_0,c)$ such that
\begin{equation}
 a_0k\le \frac{c}{2}(2^{b_0 k}-1), \quad k\ge \bar k.
\end{equation}
Hence, we have
\begin{equation}
\sum_{k=0}^{\bar k}  2^{a_0 k}e^{-c2^{b_0 k}\tau}\leq 
e^{-c\tau}\sum_{k=0}^{\bar k}  2^{a_0 k}\le  C(a_0,b_0,c)e^{-c\tau},
\end{equation}
and 
\begin{equation}
\sum_{k>\bar k}  2^{a_0 k}e^{-c2^{b_0 k}\tau}\le
e^{-c\tau}\sum_{k>\bar k}^\infty  2^{a_0 k}e^{-c(2^{b_0 k}-1)\tau} \le e^{-c\tau}
\sum_{k>\bar k}^\infty  e^{-\frac{c}{2}(2^{b_0 k}-1)} 
\le  C(a_0,b_0,c)e^{-c\tau},
\end{equation}
 where we have used the definition of $\bar k$ and the fact that $\tau\ge 1$. 
\end{proof}

\section{Proof of Lemma \ref{L:perturb}}
\label{SS:A2}

\begin{proof}
    We denote by
\begin{equation}
        \bar\beta :=\mu_{y,\sigma}(B):= \frac{1}{(2\pi \sigma^2)^{m/2}}\int_B e^{-\|x-y\|^2/2\sigma^2} dx,
\end{equation}
and will show that $\bar\beta<\frac{1}{2}$. 
Observe that   
\be
        1-\bar\alpha = \frac{1}{(2\pi \sigma^2)^{m/2}}\int_{B^c} e^{-\|x\|^2/2\sigma^2} dx,\quad 1-\bar\beta = \frac{1}{(2\pi \sigma^2)^{RR2}}\int_{B^c} e^{-\|x-y\|^2/2\sigma^2}dx.
 \ee
 Since
    \be 
   \|y\|^2= \|x\|^2-\|x-y\|^2 -2\langle x-y,y\rangle,
    \ee 
    and $\langle z,y\rangle$ is an odd function of $z$, we have   
    \begin{eqnarray}
    \frac{\|y\|^2}{2\sigma^2} &= &\frac{1}{(2\pi \sigma^2)^{m/2}}\int_{\mathbb{R}^d} \frac{1}{2\sigma^2}\left(\|x\|^2- \|x-y\|^2\right)e^{-\|x-y\|^2/2\sigma^2} dx\nonumber \\
        &=& \frac{1}{(2\pi \sigma^2)^{m/2}}\int_{\mathbb{R}^d} -\ln\left(\frac{\exp(-\|x\|^2/2\sigma^2)}{\exp(-\|x-y\|^2/2\sigma^2)}\right)e^{-\|x-y\|^2/2\sigma^2} dx.
    \end{eqnarray}
    We divide the last  integral   into integrals over the sets $B$ and $B^c$ and use Jensen's inequality and the convexity of $(-\ln x)$ to obtain
    \begin{eqnarray}
    \label{y}
        \frac{\|y\|^2}{2\sigma^2} &=& \bar\beta\int_{B} -\ln\left(\frac{\exp(-\|x\|^2/2\sigma^2)}{\exp(-\|x-y\|^2/2\sigma^2)}\right)\frac{e^{-\|x-y\|^2/2\sigma^2}}{(2\pi \sigma)^{m/2}\bar\beta} dx 
        \nonumber\\\nonumber\\
        &+& (1-\bar\beta)\int_{B^c} -\ln\left(\frac{\exp(-\|x\|^2/2\sigma^2)}{\exp(-\|x-y\|^2/2\sigma^2)}\right)\frac{e^{-\|x-y\|^2/2\sigma^2}}{(2\pi \sigma)^{m/2}(1-\bar\beta)} dx
        \nonumber\\\nonumber\\
        &\geq& -\bar\beta\ln\left(\frac{\bar\alpha}{\bar\beta}\right)-(1-\bar\beta)\ln\left(\frac{1-\bar\alpha}{1-\bar\beta}\right).
    \end{eqnarray}
    Now, if   $\bar{\beta} \geq 1/2$, we will show that $\|y\|^2\geq -\sigma^2\ln(5\bar\alpha)$ which would contradict the assumptions of the lemma.
    
    Note that 
    the first term above is lower bounded by
   \be
        -\bar\beta\ln\left(\frac{\bar\alpha}{\bar\beta}\right) \geq -\frac{1}{2}\ln(2\bar\alpha).
 \ee
    On the other hand, the second term is lower bounded by
  \be
        -(1-\bar\beta)\ln\left(\frac{1-\bar\alpha}{1-\bar\beta}\right) \geq (1-\bar\beta)\ln\left(1-\bar\beta\right) \geq \min_{0 < t < 1}\left\{ t\ln(t)\right\} = -e^{-1}.
\ee
 We use these bounds in \eref{y} to obtain that
\be
        \|y\|^2 \geq -\sigma^2[\ln(2\bar\alpha) + 2e^{-1}] = -\sigma^2\ln(2e^{2/e}\bar\alpha) \geq -\sigma^2\log(5\bar\alpha),
\ee
and this completes the proof of the lemma.
\end{proof}

\end{document}